\documentclass[9pt,a4paper]{amsart}

\usepackage[applemac]{inputenc}   
\usepackage[T1]{fontenc}      
\usepackage{geometry}         
\usepackage[english]{babel}
\usepackage{amsthm}
\usepackage{amsmath}
\usepackage{amsfonts}
\usepackage{amssymb}
\usepackage{amscd}
\usepackage[all]{xy}
\usepackage[pdftex]{hyperref}
\usepackage{mathrsfs}

\newtheorem*{theorem*}{Theorem}
\newtheorem{theorem}{Theorem}[section]
\newtheorem{proposition}[theorem]{Proposition}
\newtheorem{lemma}[theorem]{Lemma}
\newtheorem{definition}[theorem]{Definition}
\newtheorem{corollary}[theorem]{Corollary}
\theoremstyle{definition} \newtheorem*{remark*}{Remark}
\theoremstyle{theorem} \newtheorem*{conjecture*}{Conjecture}
\theoremstyle{definition} 
\theoremstyle{definition} \newtheorem{example}[theorem]{Example}
\theoremstyle{theorem} 
\begin{document}
\title{Rack homology and conjectural Leibniz homology}
\author{Simon Covez}


\begin{abstract} This article presents results being consistent with conjectures of J.-L. Loday about the existence and properties of a \textit{Leibniz homology} for groups. Introducing \textit{$\mathbf{L}$-sets} we prove that \textit{(pointed) rack homology} has properties this conjectural Leibniz homology should satisfy, namely the existence of a coZinbiel coalgebra structure on rack homology and the existence of a non trivial natural cocommutative coalgebra morphism from the rack homology of a group to its Eilenberg-MacLane homology. The end of the paper treats the particular cases of the linear group and of abelian groups. We prove the existence of a \textit{connected coZinbiel-associative bialgebra} structure on their rack homology.
\end{abstract}

\maketitle

\section*{Introduction}
\subsection*{Chevalley-Eilenberg homology and Leibniz homology} The Chevalley-Eilenberg homology is the natural homology theory associated to Lie algebras (cf. \cite{CartanEilenberg}). The Koszul dual of the operad $\mathcal{L}ie$ encoding Lie algebras is the operad $\mathcal{C}om$ encoding commutative algebras. Therefore the Chevalley-Eilenberg homology of a Lie algebra is naturally provided with a cocommutative coalgebra structure (cf. \cite{LodayVallette}).
\begin{align*}
\mathrm{H}_{\bullet}(-,\Bbbk) : \mathbf{Lie} \to \mathbf{Com}^c
\end{align*}
Given a Lie algebra $\mathfrak{g}$ this homology theory is the homology of a chain complex whose underlying graded vector space is the exterior algebra $\Lambda(\mathfrak{g})$. A fondamental remark due to J.-L. Loday is that using the antisymmetry of the Lie bracket it is possible to rewrite the differential in such a way that the relation $\mathrm{d}^2 = 0$ is a consequence of the \textit{Leibniz relation} only.
\begin{align*}
[[x,y],z] = [x,[y,z]] + [[x,z],y]
\end{align*}
As a consequence the differential on the Chevalley-Eilenberg chain complex lifts up to a differential $\mathrm{dL}$ on the tensor vector space $\mathrm{T}(\mathfrak{g})$. It defines a new chain complex $(\mathrm{T}(\mathfrak{g}),\mathrm{dL})$, and so a new homology theory called \textit{Leibniz homology} (cf. \cite{LodayCyclic}). Because the Leibniz relation is the only relation involved in the definition of $\mathrm{dL}$, this complex $(\mathrm{T}(\mathfrak{g}),\mathrm{dL})$ is defined for a category of algebras containing the category of Lie algebras. These algebras have been dubbed \textit{Leibniz algebras} by J.-L. Loday. More precisely a Leibniz algebra is a vector space $\mathfrak{g}$ provided with a bilinear map $[-,-] : \mathfrak{g} \times \mathfrak{g} \to \mathfrak{g}$ called Leibniz bracket (or shortly bracket) satisfying the Leibniz relation. The operad $\mathcal{L}eib$ encoding Leibniz algebras being Koszul dual to the operad $\mathcal{Z}inb$ encoding Zinbiel \footnote{Zinbiel algebras are Koszul dual to Leibniz algebras and are sometimes called dual Leibniz algebras; the word Zinbiel is Leibniz spelled backward} algebras, the Leibniz homology of a Leibniz algebra is naturally provided with a coZinbiel coalgebra structure.
\begin{align*}
\mathrm{HL}_{\bullet}(-,\Bbbk) : \mathbf{Leib} \to \mathbf{Zinb^c}
\end{align*}
At the operad level the morphism $\mathcal{L}eib \to \mathcal{L}ie$ induces a morphism $\mathcal{C}om \to \mathcal{Z}inb$, so there is a commutative diagram:
$$
\xymatrix{
\mathbf{Leib} \ar[rr]^{\mathrm{HL}_{\bullet}(-,\Bbbk)} & & \mathbf{Zinb}^c \ar[r] & \mathbf{Com}^c
\\
\mathbf{Lie} \ar@{^{(}->}[u] \ar[rrru]_{\mathrm{H}_{\bullet}(-,\Bbbk)}
}
$$
Moreover the canonical projection $\mathrm{T}(\mathfrak{g}) \twoheadrightarrow \Lambda(\mathfrak{g})$ induces a natural map of cocommutative algebras in homology $\mathrm{HL}_{\bullet}(\mathfrak{g},\Bbbk) \to \mathrm{H}_{\bullet}(\mathfrak{g},\Bbbk)$ fitting into a long exact sequence.
\begin{align*}
\cdots \to \mathrm{H}_{n}^{rel}(\mathfrak{g},\Bbbk) \to \mathrm{HL}_{n}(\mathfrak{g},\Bbbk) \to \mathrm{H}_{n}(\mathfrak{g},\Bbbk) \to \mathrm{H}_{n}^{rel}(\mathfrak{g},\Bbbk) \to \cdots 
\end{align*}
\subsubsection*{Chevalley-Eilenberg homology and Leibniz homology of abelian Lie algebras \textnormal{(cf. \cite{LodayEns})}} The Chevalley-Eilenberg and Leibniz differentials of an abelian Lie algebra are trivial. Therefore $\mathrm{H}_{n}(\mathfrak{g},\Bbbk) \simeq \Lambda^n(\mathfrak{g})$ and $\mathrm{HL}_{n}(\mathfrak{g},\Bbbk) \simeq \mathrm{T}^n(\mathfrak{g})$ for all $n \in \mathbb{N}$, and the morphism from Leibniz homology to Chevalley-Eilenberg homology is the canonical projection $\mathrm{T}(\mathfrak{g}) \twoheadrightarrow \Lambda(\mathfrak{g})$.

\subsubsection*{Chevalley-Eilenberg homology and Leibniz homology of the Lie algebra of matrices $\mathfrak{gl}(A)$ \textnormal{(cf. \cite{LodayCyclic})}} Given a unital associative algebra $A$ over a field of characteristic $0$, the Chevalley-Eilenberg and the Leibniz homologies of the Lie algebra of matrices $\mathfrak{gl}(A)$ are provided with more structure. 
\par
There is a product on $\mathfrak{gl}(A)$, called the direct sum of matrices, which induces a connected commutative graded Hopf algebra structure on the Chevalley-Eilenberg homology of $\mathfrak{gl}(A)$. As a consequence if we combine the Hopf-Borel theorem and the Loday-Quillen-Tsygan theorem, then we obtain the following isomorphism of Hopf algebras
\begin{align*}
\mathrm{H}_{\bullet}(\mathfrak{gl}(A),\Bbbk) \simeq \mathrm{S}\big(\mathrm{HC}_{\bullet-1}(A)\big)
\end{align*}
where $\mathrm{HC}_{\bullet}(A)$ is the \textit{cyclic homology} of $A$. 
\par
The direct sum of matrices induces a connected graded coZinbiel-associative bialgebra structure on the Leibniz homology of $\mathfrak{gl}(A)$. There exists a structure theorem for such type of bialgebras (cf. \cite{Burgunder_graph_complex_and_Leibniz_homology}). As a consequence if we combine this structure theorem and the Loday-Cuvier theorem, then we obtain the following isomorphism of coZinbiel-associative bialgebras
\begin{align*}
\mathrm{HL}_{\bullet}(\mathfrak{gl}(A),\Bbbk) \simeq \mathrm{T}\big(\mathrm{HH}_{\bullet-1}(A)\big)
\end{align*} 
where $\mathrm{HH}_{\bullet-1}(A)$ is the \textit{Hochschild homology} of $A$.

\subsection*{Eilenberg-MacLane homology and conjectural Leibniz homology for groups} The \textit{Eilenberg-MacLane homology} is the natural homology theory associated to groups (cf. \cite{CartanEilenberg}). This homology theory is naturally associated with a cocommutative coalgebra structure.
\begin{align*}
\mathrm{H}_{\bullet}(-,\Bbbk) : \mathbf{Grp} \to \mathbf{Com}^c
\end{align*}
Lie algebras being the linearized objects associated to groups, the existence and properties of the Leibniz homology theory for Lie algebras and Leibniz algebras led J.-L. Loday to state this conjecture.
\begin{conjecture*}[J.-L. Loday \cite{LodayEns,LodayConjectural}] There exists a \textnormal{Leibniz homology theory} defined for groups which is naturally endowed with a coZinbiel coalgebra structure. 
\begin{align*}
\mathrm{HL}_{\bullet}(-,\Bbbk) : \mathbf{Grp} \to \mathbf{Zinb}^c
\end{align*}
This Leibniz homology is related to the usual group homology by a natural morphism of cocommutative algebras.
\begin{align*}
\mathrm{HL}_{\bullet}(G,\Bbbk) \to \mathrm{H}_{\bullet}(G,\Bbbk)
\end{align*}
This Leibniz homology is the natural homology theory of mathematical objects called \textnormal{coquecigrues} whose groups carry naturally the structure.
$$
\xymatrix{
\mathbf{Coquecigrues} \ar[rr]^{\quad \mathrm{HL}_{\bullet}(-,\Bbbk)} & & \mathbf{Zinb}^c \ar[r] & \mathbf{Com}^c 
\\
\mathbf{Grp} \ar@{^{(}->}[u] \ar[rrru]_{\quad \mathrm{H}_{\bullet}(-,\Bbbk)} & &
}
$$
\end{conjecture*}

\subsubsection*{Eilenberg-MacLane homology and conjectural Leibniz homology of abelian groups} Eilenberg-MacLane homology of an abelian group is well known (cf. \cite{Brown_Cohomology_of_groups}). It is provided with a graded Hopf algebra structure, and for a field $\Bbbk$ of characteristic $0$ there is an isomorphism $\mathrm{H}_n(G,\Bbbk) \simeq \Lambda^n(G \otimes \Bbbk)$ for all $n \in \mathbb{N}$. Abelian Lie algebras being linearized objects associated to abelian groups, the similar properties satisfied by the Chevalley-Eilenberg homology and the Leibniz homology of an abelian Lie algebra led J.-L. Loday to state the following conjecture.
\begin{conjecture*}[J.-L. Loday \cite{LodayEns,LodayConjectural}] The Leibniz homology of an abelian group $G$ is provided with a connected coZinbiel-associative bialgebra structure. Moreover for all $n \in \mathbb{N}$ there is an isomorphism $\mathrm{HL}_{n}(G,\Bbbk) \simeq \mathrm{T}^n(G \otimes \Bbbk)$, and if the characteristic of $\Bbbk$ is $0$, then the natural morphism from $\mathrm{HL}_n(G,\Bbbk)$ to $\mathrm{H}_n(G,\Bbbk)$ is the canonical projection $\mathrm{T}^n(G \otimes \Bbbk) \to \Lambda^n(G \otimes \Bbbk)$.
\end{conjecture*}

\subsubsection*{Eilenberg-MacLane homology and conjectural Leibniz homology of the linear group $\mathrm{GL}(R)$} The Eilenberg-MacLane homology of the (infinite) linear group $\mathrm{GL}(R)$ is provided with a connected commutative graded Hopf algebra structure (cf. \cite{LodayCyclic}). The Lie algebra of matrices $\mathfrak{gl}(A)$ being the linearized object associated to the linear group, the similar propreties satisfied by the Chevalley-Eilenberg homology and the Leibniz homology of $\mathfrak{gl}(A)$ led J.-L. Loday to state the following conjecture.
\begin{conjecture*}[J.-L. Loday \cite{LodayEns,LodayConjectural}] The Leibniz homology of the linear group $\mathrm{GL}(R)$ is provided with a connected coZinbiel-associative bialgebra structure. 
\end{conjecture*}

\subsection*{Results} A natural candidate for this conjectural Leibniz homology is the natural homology theory of objects integrating Leibniz algebras. A \textit{rack} is a mathematical object which encapsulates some properties of the conjugation in a group (cf. \cite{Joyce_A_classifying_invariant_of_knots}). There are at least two ways to construct a Leibniz algebra from a rack, a geometric one due to M.K. Kinyon (\cite{Kinyon}) and an algebraic one due to S. Dansco (\cite{Dancso}), and reciprocally any Leibniz algebra integrates into a \textit{local Lie rack} (cf. \cite{CovezIntegration}). As a consequence it is natural to suspect the homology theory of racks to be linked to this conjectural Leibniz homology theory.
\vskip 0.2cm
In this article we prove that rack homology satisfies most of the properties a Leibniz homology for groups should satisfy. These results are the contents of the following theorems. 
\subsubsection*{Algebraic structure on rack homology and relation with Eilenberg-MacLane homology}
\begin{theorem*}
The \textnormal{rack homology} of a rack $X$ is naturally provided with a coZinbiel coalgebra structure
\begin{align*}
\mathrm{HR}_{\bullet}(-,\Bbbk) : \mathbf{Rack} \to \mathbf{Zinb}^c,
\end{align*}
and there is a natural morphism $\mathrm{S}_{\bullet} : \mathrm{HR}_{\bullet}(G,\Bbbk) \to \mathrm{H}_{\bullet}(G,\Bbbk)$ of cocommutative coalgebra from the rack homology of a group $G$ to its Eilenberg-MacLane homology fitting in a long exact sequence.
\begin{align*}
\cdots \to \mathrm{H}^{rel}_{n+1}(G,\Bbbk) \to \mathrm{HR}_n(G,\Bbbk) \stackrel{\mathrm{S}_n}{\to} \mathrm{H}_n(G,\Bbbk) \to \mathrm{H}_n^{rel}(G,\Bbbk) \to \cdots 
\end{align*}
\end{theorem*}
\subsubsection*{Rack homology of abelian groups}
\begin{theorem*} The rack homology of an abelian group $(G,+)$ is provided with a connected coZinbiel-associative bialgebra structure where the associative product is induced by the rack morphism $+ : G \times G \to G$. Moreover for all $n \in \mathbb{N}$ there is an isomorphism $\mathrm{HR}_n(G,\Bbbk) \simeq \mathrm{T}^n(\Bbbk[G\setminus\{0\}])$, and if the characteristic of $\Bbbk$ is $0$, then the natural morphism from $\mathrm{HR}_n(G,\Bbbk)$ to $\mathrm{H}_n(G,\Bbbk)$ is the canonical map $\mathrm{T}^n(\Bbbk[G\setminus\{0\}]) \to \Lambda^n(G \otimes \Bbbk)$.
\end{theorem*}
\subsubsection*{Rack homology of the linear group $\mathrm{GL}(R)$}
\begin{theorem*}
The rack homology of the linear group $\mathrm{GL}(R)$ is provided with a connected coZinbiel-associative bialgebra structure where the associative product is induced by the direct sum of matrices.
\end{theorem*}

The plan for this article is the following.

\subsection*{Section 1 : (co)Dendriform and (co)Zinbiel (co)algebras} This section is based on \cite{LodayVallette} and \cite{Burgunder_graph_complex_and_Leibniz_homology}. It recalls definitions about (co)dendriform and (co)Zinbiel (co)algebras, together with definitions and a structure theorem about coZinbiel-associative bialgebras.

\subsection*{Section 2 : $\mathbf{L}$-sets} This section is the core of our paper where we define the new notion of \textit{$\mathbf{L}$-sets}. The category of \textit{$\mathbf{L}$-sets}, denoted $\mathbf{LSet}$, is a full subcategory of the category $\mathbf{cSet}$ of cubical sets. First we prove the existence of left and right adjoint $\mathrm{L},\Gamma : \mathbf{LSet} \to \mathbf{cSet}$ to the forgetful functor $\mathrm{U} : \mathbf{LSet} \to \mathbf{cSet}$.
\begin{align*}
\Gamma \vdash \mathrm{U} \dashv \mathrm{L} 
\end{align*}
By construction there exists of a long exact sequence relating the homology of a cubical set $\mathrm{X}$ and the homology of the $\mathbf{L}$-set $\mathrm{L(X)}$ naturally associated to it.
\begin{align*}
\cdots \to \mathrm{H}^{rel}_{n+1}(X,\Bbbk) \to \mathrm{H}_{n}(\mathrm{L(X)},\Bbbk) \to \mathrm{H}_{n}(\mathrm{X},\Bbbk) \to \mathrm{H}^{rel}_{n}(\mathrm{X},\Bbbk) \to \cdots
\end{align*}
We finish this section with the proof of the existence of a coZinbiel coalgebra structure on the homology of any $\mathbf{L}$-set (Theorem \ref{Theorem : Zinbiel up to homotopy coalgebra structure on Leibniz homology}) using the method of the acyclic models. 

\subsection*{Section 3 : $\mathbf{L}$-homology of groups and rack homology} In this section we compute the $\mathbf{L}$-set $\mathrm{L}(\mathrm{N}^{\square}G)$ associated to the \textit{cubical nerve} of a group $G$. The importance of the cubical nerve is that its homology is isomorphic to the Eilenberg-MacLane homology of the group. We prove that $\mathrm{L}(\mathrm{N}^{\square}G)$ is isomorphic to the \textit{nerve of the rack $G$} (Proposition \ref{Theorem : bijection from L(N(G,M)) to MGn}). This result implies that the $\mathbf{L}$-homology of the cubical set $\mathrm{L}(\mathrm{N}^{\square}G)$ is exactly the rack homology of $G$. With the results proved in the previous section, we deduce the existence of a coZinbiel coalgebra structure on rack homology (Theorem \ref{Theorem : Zinbiel up to homotopy coalgebra structure on rack homology}) and the existence of a morphism from the rack homology of a group to its Eilenberg-MacLane homology fitting into a long exact sequence (Theorem \ref{Theorem : long exact sequence linking HR and H}). 

\subsection*{Section 4 : Rack homology of abelian groups} In this section we focus on the particular case of abelian groups. First the definition of the graded Hopf algebra structure on the Eilenberg-MacLane homology of an abelian group is recalled, together with the computation of the Eilenberg-MacLane homology groups of an abelian group. Then we prove the existence of a coZinbiel-associative bialgebra structure on the rack homology of an abelian group, and we compute the rack homology groups of an abelian group. 

\subsection*{Section 5 : Rack homology of the linear group} In this section we focus on the particular case of the linear group $\mathrm{GL}(R)$. First we recall the definition of the graded Hopf algebra structure on the Eilenberg-MacLane homology of $\mathrm{GL}(R)$. Then we prove the existence of a coZinbiel-associative bialgebra structure on its rack homology (Theorem \ref{Theorem : coZinbiel-associative bialgebra structure on the rack homology of the linear group}).

\subsection*{Appendix A: Acyclic models} This appendix is a summary of \cite{EilenbergMacLane}. It recalls definitions and theorems of the Eilenberg-MacLane theory about \textit{acyclic models}. 

\subsection*{Appendix B: Simplicial and cubical sets} This appendix recalls definitions and properties of simplicial and cubical sets. Especially we remind the construction of a cocommutative coalgebra structure on the homology of a cubical set using the method of acyclic models.

\section{(co)Dendriform and (co)Zinbiel (co)algebras}

\subsection{Shuffle} For all $n \in \mathbb{N}^*$ let $\mathbb{S}_n$ be the group of permutations of the set $\{1,\dots,n\}$. For all $p,q \in \mathbb{N}^*$ let $\mathrm{Sh}_{p,q}$ be the subset of elements $\sigma \in \mathbb{S}_{p+q}$ satisfying
\begin{align*}
\sigma(1) < \cdots < \sigma(p) \, \text{ and } \, \sigma(p+1) < \cdots < \sigma(p+q).  
\end{align*}
Such an element $\sigma$ is called a \textit{(p,q)-shuffle}. Remark that $\sigma(1) = 1$ or $\sigma(1) = p+1$. The subset of $(p,q)$-shuffles satisfying $\sigma(1) = 1$ (resp. $\sigma(1) = p+1$) is denoted by $\mathrm{Sh}_{p,q}^1$ (resp. $\mathrm{Sh}_{p,q}^{p+1}$).
\par
There is a bijection $\mathrm{Sh}_{p,q} \stackrel{\iota}{\simeq} \mathrm{Sh}_{q,p}$ given by
\begin{align*}
\iota(\sigma)(k) &:= 
\left\{\begin{array}{ll}
\sigma(k+p) &  \text{if } 1 \leq k \leq q,\\
\sigma(k-q) & \text{if } q+1 \leq k \leq p+q.
\end{array}\right.
\end{align*}
\par
For all $p,q,r \in \mathbb{N}^*$ let $\mathrm{Sh}_{p,q,r}$ be the subset of elements $\sigma \in \mathbb{S}_{p+q+r}$ satisfying
\begin{align*}
\sigma(1) < \cdots < \sigma(p) \, \text{ and } \, \sigma(p+1) < \cdots < \sigma(p+q) \, \text{ and } \, \sigma(p+q+1) < \cdots < \sigma(p+q+r).
\end{align*}
Such an element $\sigma$ is called a \textit{(p,q,r)-shuffle}.
\par
For all $p,q,r \in \mathbb{N}^*$ there are bijections $\mathrm{Sh}_{p+q,r} \times \mathrm{Sh}_{p,q} \stackrel{\alpha}{\simeq} \mathrm{Sh}_{p,q,r} \stackrel{\beta}{\simeq} \mathrm{Sh}_{p,q+r} \times \mathrm{Sh}_{q,r}$ given by 
\begin{align*}
\alpha(\sigma,\gamma)(k) &:= 
\left\{\begin{array}{ll}
\sigma(\gamma(k)) &  \text{if } 1 \leq k \leq p+q,\\
\sigma(k) & \text{if } p+q+1 \leq k \leq p+q+r.
\end{array}\right.
\\
\beta(\sigma,\gamma)(k) &:=
\left\{\begin{array}{ll}
\sigma(k) &  \text{if } 1 \leq k \leq p,\\
\sigma(p+\gamma(k-p)) & \text{if } p+1 \leq k \leq p+q+r.
\end{array}\right.
\end{align*} 

\subsection{Dendriform (co)algebra (\cite{LodayVallette})} A \textit{(graded) dendriform algebra} is a graded vector space $A$ endowed with two products $\prec,\succ : A \otimes A \to A$ satisfying the following relations called \textit{dendriform relations}:
\begin{align*}
(x \prec y) \prec z &= x \prec (y \prec z + y \succ z),\\
(x \succ y) \prec z &= x \succ (y \prec z),\\
x \succ (y \succ z) &= (x \succ y + x \prec y) \succ z.
\end{align*}
A dendriform algebra $(A,\prec,\succ)$ is said to be \textit{unital} if there exists an element $1 \in A$ such that $1 \prec x = x \succ 1 = 0$ and $1 \succ x = x \prec 1 = x$ for all $x \in A$. Note that $1 \prec 1$ and $1 \succ 1$ are not defined. Given a dendriform algebra $(A,\prec,\succ)$, the product $\star := \prec + \succ$ on $A$ is associative. Therefore there exists a functor between categories of algebras $\mathbf{Dend} \to \mathbf{As}$.
\vskip 0.2cm
By duality we get the definition of a codendriform coalgebra. A \textit{(graded) codendriform coalgebra} is a graded vector space $C$ endoced with two coproducts $\Delta_{\prec},\Delta_{\succ} : C \to C \otimes C$ satisifying the following relation called \textit{codendriform relation}:
\begin{align*}
(\Delta_{\prec} \otimes \mathrm{id}) \circ \Delta_{\prec} &= (\mathrm{id} \otimes \Delta_{\prec} + \mathrm{id} \otimes \Delta_{\succ}) \circ \Delta_{\prec},\\
(\Delta_{\succ} \otimes \mathrm{id}) \circ \Delta_{\prec} &= (\mathrm{id} \otimes \Delta_{\prec}) \circ \Delta_{\succ},\\
(\mathrm{id} \otimes \Delta_{\succ}) \circ \Delta_{\succ} &= (\Delta_{\prec} \otimes \mathrm{id} +  \Delta_{\succ} \otimes \mathrm{id}) \circ \Delta_{\succ}.
\end{align*}
A codendriform coalgebra is said to be \textit{counital} if there exists a linear map $c : C \to \Bbbk$ such that $(c \otimes \mathrm{id}) \circ \Delta_{\prec} = (\mathrm{id} \otimes c) \circ \Delta_{\succ} = 0$ and $(c \otimes \mathrm{id}) \circ \Delta_{\succ} = (\mathrm{id} \otimes c) \circ \Delta_{\prec} = \mathrm{id}$. Note that $(c \otimes c) \circ \Delta_{\prec}$ and $(c \otimes c) \circ \Delta_{\succ}$ are not defined. 
\vskip 0.2cm
Given a codendriform coalgebra $(A,\Delta_{\prec},\Delta_{\succ})$, the coproduct $\Delta := \Delta_{\prec} + \Delta_{\succ}$ on $A$ is coassociative. Therefore there exists a functor between categories of coalgebras $\mathbf{Dend}^c \to \mathbf{As}^c$.
\subsection{(co)Zinbiel (co)algebra (\cite{LodayVallette})} A \textit{(graded) Zinbiel algebra} is a graded vector space $A$ endowed with a product $\prec : A \otimes A \to A$ satisfying the following relation called \textit{(graded) Zinbiel relation}:
\begin{align*}
(x \prec y) \prec z = x \prec (y \prec z + (-1)^{|x||y|} z \prec y).
\end{align*}
A Zinbiel algebra $(A,\prec)$ is said to be \textit{unital} if there exists an element $1 \in A$ such that the following is verified $1 \prec x = 0$ and $x \prec 1 = x$ for all $x \in A$. Note that $1 \prec 1$ is not defined. Given a Zinbiel algebra $(A,\prec)$, the products $\prec$ and $\succ := \tau \circ \prec$ define a dendriform algebra structure on $A$. Therefore there exists a functor between categories of algebras $\mathbf{Zinb} \to \mathbf{Dend}$. The product $\star := \prec + \prec \circ \, \tau$ on $A$ is associative and commutative. Therefore there exists a functor between categories of algebras $\mathbf{Zinb} \to \mathbf{Com}$.
\vskip 0.2cm
By duality we get the definition of a coZinbiel coalgebra. A \textit{(graded) coZinbiel coalgebra} is a graded vector space $C$ endowed with a coproduct $\Delta_{\prec} : C \to C \otimes C$ satisfying the following relation called \textit{coZinbiel relation}:
\begin{align*}
(\Delta_{\prec} \otimes \mathrm{id}) \circ \Delta_{\prec} = \big(\mathrm{id} \otimes \Delta_{\prec} + \mathrm{id} \otimes (\tau \circ \Delta_{\prec})\big) \circ \Delta_{\prec}.
\end{align*}
\par
A coZinbiel coalgebra $(C,\Delta_{\prec})$ is said to be \textit{counital} if there exists a linear map $c : C \to \Bbbk$ such that the relations $(c \otimes \mathrm{id}) \circ \Delta_{\prec} = 0$ and $(\mathrm{id} \otimes c) \circ \Delta_{\prec} = \mathrm{id}$ are satisfied. Note that $(c \otimes c) \circ \Delta_{\prec}$ is not defined. Given a connected coZinbiel coalgebra $(C,\Delta_{\prec},c)$ the augmentation ideal $\overline{C} := \mathrm{Ker}(c)$ is a (non counital) coZinbiel coalgebra for the reduced coproduct $\Delta_{\prec}^-(x) = \Delta_{\prec}(x) - x \otimes 1$. Conversely, given a non counital coZinbiel coalgebra $(\overline{C},\Delta^-_{\prec})$, the augmented vector space $C := \overline{C} \oplus \Bbbk$ provided with the augmented coproduct $\Delta_{\prec} := \Delta_{\prec}^-(x) + x \otimes 1$ is a counital coZinbiel coalgebra.
\vskip 0.2cm
Let $(\overline{C},\Delta^-_{\prec},)$ be a coZinbiel coalgebra. A filtration of $\overline{C}$ is defined by 
\begin{align*}
F_1\overline{C} &:= \{x \in \overline{C} \, | \, \Delta_{\prec}^-(x) = 0\} \\
F_r\overline{C} &:= \{x \in \overline{C} \, | \, \Delta_{\prec}^-(x) \in F_{r-1}\overline{C} \otimes F_{r-1}\overline{C} \} \,  \text{ for all } \, r >1
\end{align*} 
The coZinbiel coalgebra $\overline{C}$ is said to be \textit{connected} (or \textit{conilpotent}) if $\overline{C} = \bigcup_{r \geq 1} F_r\overline{C}$. The space of \textit{primitive elements of $\overline{C}$} is the first piece of the filtration.
\begin{align*}
\mathrm{Prim}(\overline{C}) := F_1\overline{C}
\end{align*} 
\par
Given a coZinbiel coalgebra $(A,\Delta_{\prec})$, the coproducts $\Delta_{\prec}$ and $\Delta_{\succ} = \tau \circ \Delta_{\prec}$ define a codendriform coalgebra structure on $A$. Therefore there exists a functor between categories of algebras $\mathbf{Zinb}^c \to \mathbf{Dend}^c$. The coproduct $\Delta := \Delta_{\prec} + \tau \circ \Delta_{\prec}$ on $A$ is coassociative and cocommutative. Therefore there exists a functor between categories of algebras $\mathbf{Zinb}^c \to \mathbf{Com}^c$.
\begin{example}
Let $V$ be a graded vector space. Let us define a coZinbiel coproduct on the reduced tensor vector space $\overline{\mathrm{T}}(V)$ by:
\begin{align} \displaystyle
\Delta_{\prec}(x_1 \cdots x_n) := \sum_{p+q = n} \sum_{\sigma \in \mathrm{Sh}_{p,q}^1} \epsilon(\sigma) \, x_{1}x_{\sigma(2)} \cdots x_{\sigma(p)} \otimes x_{\sigma(p+1)} \cdots x_{\sigma(p+q)} \label{half-shuffle coproduct}  
\end{align} 
This coproduct is called the \textit{half-shuffle coproduct}.
\end{example}
\subsection{coZinbiel-associative bialgebra (\cite{Burgunder_graph_complex_and_Leibniz_homology})} A \textit{(graded) coZinbiel-associative bialgebra $(H,\star,\Delta_{\prec})$} is a graded vector space $H$ endowed with a counital coZinbiel coproduct $\Delta_{\prec}: H \to H \otimes H$ and an associative product $\star : H \otimes H \to H$ satisfying the following compatibility relation called \textit{semi-Hopf relation}:
\begin{align*}
\Delta_{\prec} \circ \star = \star_{\otimes} \circ (\Delta \otimes \Delta_{\prec}).  
\end{align*}
\begin{example} \label{Example : Bialgebra Zinb-As T(V)} The graded vector space $\mathrm{T}(V)$ with the concatenation product and the half-shuffle coproduct \eqref{half-shuffle coproduct}.  
\end{example}
\par
\subsection{Structure theorem for coZinbiel-associative bialgebras (\cite{Burgunder_graph_complex_and_Leibniz_homology})}\label{Theorem : Hopf-Borel} The \textit{Hopf-Borel theorem} is a structure theorem for (graded) Hopf algebras. This theorem states that a connected commutative graded Hopf algebra over a field of characteristic $0$ is free and cofree over its primitive part. As a consequence such a Hopf algebra is isomorphic to $(\mathrm{S}(V),\cdot,\Delta_{\mathrm{Sh}})$ where $V$ is its primitive part, $\cdot$ the canonical symmetric product on $\mathrm{S}(V)$ and $\Delta_{\mathrm{Sh}}$ the \textit{shuffle coproduct}.
\vskip 0.2cm
In the case of a coZinbiel-associative bialgebra a similar structure theorem holds. 
\begin{theorem}{\cite{Burgunder_graph_complex_and_Leibniz_homology}} \label{Theorem : Structure theorem for coZinbiel-associative bialgebras}
A connected coZinbiel-associative bialgebra (over a field of any characteristic) is free and cofree over its primitive part. 
\end{theorem}
As a consequence a connected coZinbiel-associative bialgebra is isomorphic as a bialgebra to $\mathrm{T}(V)$ where $V$ is its primitive part (cf. Example \ref{Example : Bialgebra Zinb-As T(V)}).

\section{$\mathbf{L}$-sets}
In this section we introduce \textit{$\mathbf{L}$-sets}. These objects are cubical sets (cf. \ref{Definition : cubical sets}) for which the cocommutative coalgebra structure on the homology is induced by a coZinbiel coalgebra structure. 
\subsection{$\mathbf{L}$-sets} A \textit{$\mathbf{L}$-set} is a cubical set $\mathrm{X}$ such that $\mathrm{X}_0 = \{\star\}$, and for all $n \in \mathbb{N}$
\begin{align*}
d_{1,0} = d_{1,1} : X_{n} \to X_{n-1}
\end{align*}
By definition a $\mathbf{L}$-set is a cubical set, thus there is a functor $\mathrm{U}$ from the category $\mathbf{LSet}$ of $\mathbf{L}$-set to the category $\mathbf{cSet}$ of cubical set
\begin{align*}
\mathrm{U} : \mathbf{LSet} \to \mathbf{cSet}.
\end{align*}
\subsection{The category $\mathbf{L}$} A $\mathbf{L}$-set can be defined as a contravariant functor from a certain category $\mathbf{L}$ to the category $\mathbf{Set}$ in the following way. Let $\mathbf{L}$ be the category with the same objects as $\square$, but with set of morphisms $\mathrm{Hom}_{\mathbf{L}}(\square_m,\square_n)$ from $\square_m$ to $\square_n$ defined as the coequalizer of $d_{1,1}$ and $d_{1,0}$.
\begin{align*}
\xymatrix{ 
\mathrm{Hom}_{\square}(\square_{m+1},\square_n) \ar@<0.5ex>[r]^{d_{1,1}} \ar@<-0.5ex>[r]_{d_{1,0}} & \mathrm{Hom}_{\square}(\square_m,\square_n) \ar@{>>}[r] &   \mathrm{coeq}(d_{1,0},d_{1,1}) =: \mathrm{Hom}_{\mathbf{L}}(\square_m,\square_n) 
 }
\end{align*}

Thanks to the universal property of a quotient category, there is a projection functor from $\square^{\mathrm{op}}$ to $\mathbf{L}^{\mathrm{op}}$ and each $\mathbf{L}$-set $\mathrm{X}$ factorizes in a unique way through the category $\mathbf{L}^{\mathrm{op}}$.
\begin{align*}
\xymatrix{
\square^{\mathrm{op}} \ar[r]^{\mathrm{X}} \ar[d] &  \mathbf{Set}
\\
\mathbf{L}^{\mathrm{op}} \ar[ru]_{\overline{\mathrm{X}}} &
}
\end{align*}
As a consequence the category of $\mathbf{L}$-set is equivalent to the category of functors from $\mathbf{L}^{\mathrm{op}}$ to $\mathbf{Set}$.
\begin{align*}
\mathbf{LSet} \simeq [\mathbf{L}^{\mathrm{op}},\mathbf{Set}]
\end{align*}

\subsection{The functors $\mathrm{L}$ and $\Gamma$} In this section we construct a left and a right adjoint to the forgetful functor $\mathrm{U}: \mathbf{LSet} \to \mathbf{cSet}$. 
\vskip0.2 cm
Let $\mathrm{L}$ be the functor from $\mathbf{cSet}$ to $\mathbf{LSet}$ defined on objects by $\mathrm{LX}_0 = \{\star\}$ and for all $n \geq 1$ 
\begin{align*} \displaystyle 
\mathrm{LX}_n &= \bigcap_{1 \leq k \leq n} \bigcap_{\epsilon,\epsilon' \in \{0,1\}^k} \mathrm{eq}(d_{1,\epsilon_1}\cdots d_{1,\epsilon_k} ; d_{1,\epsilon'_1} \cdots d_{1,\epsilon'_k}) \subseteq \mathrm{X}_n,
\end{align*}
$\displaystyle \mathrm{LX}(d_{i,\epsilon}) = (d_{i,\epsilon})_{|\mathrm{LX}_n}, \,  \mathrm{LX}(s_{i}) = (s_{i})_{|\mathrm{LX}_n}$ (well defined thanks to the cubical identities), and on morphisms by declaing that $\mathrm{Lt}_0$ is the unique map from $\star$ to $\star$ and $\mathrm{Lt}_n = \mathrm{t}_{|\mathrm{LX}_n}$ for all $n \geq 1$. There is a natural transformation $\mathrm{inc} : \mathrm{L} \to \mathrm{id}_{\mathbf{cSet}}$ from $\mathrm{L}$ to the identity functor induced by the inclusions $\mathrm{LX}_n \subseteq \mathrm{X}_n$.  

\vskip 0.2cm

Let $\Gamma$ be the functor from $\mathbf{cSet}$ to $\mathbf{LSet}$ defined on objects by $\Gamma\mathrm{X}_0 = \{\star\}$ and for all $n \geq 1$
\begin{align*}
\displaystyle \Gamma\mathrm{X}_n =  \mathrm{coeq}(d_{1,0} ; d_{1,1}) 
\end{align*} 
and $\displaystyle \mathrm{\Gamma X}(d_{i,\epsilon}) = \overline{d_{i,\epsilon}}$ (well defined thanks to the cubical idenitites), and on morphisms by declaring that $\mathrm{\Gamma t}_0$ is the unique map from $\star$ to $\star$ and $\mathrm{\Gamma t}_n = \overline{\mathrm{t}_n}$. There is a natural transformation $\mathrm{proj} : \mathrm{id}_{\mathbf{cSet}} \to \Gamma$ from the identity functor to $\Gamma$ induced by the projections $\mathrm{X}_n \twoheadrightarrow \Gamma \mathrm{X}_n$.  
\vskip 0.2cm
There are adjunctions $\mathrm{\Gamma \vdash U}$ and $\mathrm{U \dashv L}$, and the functors $\mathrm{L}$ and $\Gamma$ satisfy the relations $\mathrm{L \circ L = L}, \, \Gamma \circ \Gamma = \Gamma, \, \Gamma \circ \mathrm{L} = \mathrm{L}$ and $\mathrm{L} \circ \Gamma = \Gamma$.
As a consequence a cubical set $\mathrm{X}$ is a $\mathbf{L}$-set if and only if $\mathrm{X = LX = \Gamma X}$.

\subsection{Representable contravariant functors of $\mathbf{L}$} For all $n \in \mathbb{N}$ let us denote by $\mathrm{L}^n$ the representable functor $\mathrm{Hom}_{\mathbf{L}}(-,\square_n)$. By definition $\mathrm{L}^n = \Gamma\square^n$.
\begin{lemma} \label{Proposition : CLn representable}
The restrictions of functors $\mathrm{Q}_n$ and $\mathrm{C}_n$ (cf. \ref{Appendix : Homology of cubical sets}) to $\mathbf{LSet}$ are representable by $\Bbbk.\mathrm{L}^n$.
\end{lemma}

\begin{proof} 
Let $n \in \mathbb{N}$ be fixed and take as set of models $\mathcal{M}$ the set with one element $\{\Bbbk.\mathrm{L}^n\}$. By definition the functor $\mathrm{Q}_n$ is the composition of the functor $\Bbbk. : \mathbf{Set} \to \Bbbk\mathbf{Mod}$ and the evaluation functor $\mathrm{ev}_{\square_n} : \mathbf{LSet} \to \mathbf{Set}$
\begin{align*}
\mathrm{Q}_n := \Bbbk. \circ \mathrm{ev}_{\square_n}. 
\end{align*}
Let us define a natural transformation $\Psi$ from $\mathrm{Q}_n$ to $\mathrm{\widetilde{Q}}_n$ (cf. \ref{Appendix : Representable functor}) by the formula
\begin{align*}
\Psi_{\mathrm{X}} : \mathrm{Q}_n(\mathrm{X}) \to \mathrm{\widetilde{Q}}_n(\mathrm{X}) \, ; \, \Psi_{\mathrm{X}}(x) = (\phi_x,\overline{\mathrm{id}_{\square_n}}),
\end{align*}
where $\phi_x$ is the unique natural transformation from $\mathrm{L}^n$ to $\mathrm{X}$ such that $(\phi_x)_n(\overline{\mathrm{id}_{\square_n}}) = x$ (Yoneda's Lemma). We have $\Phi \circ \Psi = \mathrm{id}$ so $\mathrm{Q}_n : \mathbf{LSet} \to \Bbbk\mathbf{Mod}$ is representable.
\vskip 0.2cm
The representability of $\mathrm{C}_n$ is a consequence of Lemma \ref{the quotient of a representable is representable}. Indeed, let $\xi : \mathrm{Q}_n \to \mathrm{C}_n$ be the natural transformation defined by passing to the quotient, and let $\eta : \mathrm{C}_n \to \mathrm{Q}_n$ be the natural transformation defined by 
\begin{align*}
\eta_{\mathrm{X}} := (\mathrm{id} - s_1 d_{1,0}) \cdots  (\mathrm{id} - s_n d_{n,0}). 
\end{align*} 
Thanks to the cubical identities this map sends $\mathrm{D}_n(\mathrm{X})$ (cf. \ref{Appendix : Homology of cubical sets}) to $0$ and so is well defined. Moreover $\xi \circ \eta = \mathrm{id}$ then by Lemma \ref{the quotient of a representable is representable} the functor $\mathrm{C}_n$ is representable. 
\end{proof}

\begin{lemma} \label{Computation of H(Ln)}
For all $n \in \mathbb{N}$, $\mathrm{H}_{\bullet}(\mathrm{L}^n,\Bbbk) =
\left\{\begin{array}{ll}
0 & \text{if } \, \bullet \neq 1,\\
\Bbbk^{n} & \text{if } \, \bullet = 1. 
\end{array}\right.
$
\end{lemma} 

\begin{proof} Let $n \in \mathbb{N}$ be fixed. The homology of $\mathrm{L}^n$ is isomorphic to the singular homology of its geometric realization $| \mathrm{L}^n |$. This topological space is homeomorphic to the quotient of $[0,1]^n$ by the equivalence relation identifying $(d_{1,\epsilon_1} \cdots d_{1,\epsilon_k})(x)$ and $(d_{1,\epsilon'_1} \cdots d_{1,\epsilon'_k})(x)$ for all $x \in [0,1]^n, \, 1 \leq k \leq n$ and $\epsilon_i,\epsilon'_i \in \{0,1\}$. 
\vskip 0.2cm
Let $\{A_i,B_i\}_{1 \leq i n}$ be a family of subspace of $|\mathrm{L}^n|$ defined by :
\begin{itemize}
\item $A_i := p(\{(x_1,\dots,x_n) \in [0,1]^n \, | \, x_i \neq \frac{1}{2}\})$,
\item $B_i := p(\{(x_1,\dots,x_n) \in [0,1]^n \, | \, \frac{1}{4} \leq x_i \leq \frac{3}{4} \})$. 
\end{itemize}
where $p : [0,1]^n \to |\mathrm{L}^n|$ is the projection. Thanks to the properties
\begin{itemize}
\item $A_1 \cup B_1 = |\mathrm{L}^n|$,
\item $A_1 \cap B_1 \simeq \{\star_1,\star_2\}$,
\item $B_1 \simeq \{\star\}$,
\end{itemize}
the Mayer-Vietoris exact sequence applied to the covering $| \mathrm{L}^n | = A_1 \cup B_1$
\begin{align*}
\cdots \to \mathrm{H}_{p+1}(| \mathrm{L}^n |) \to \mathrm{H}_p(A_{1} \cap B_{1}) \to \mathrm{H}_p(A_{1}) \oplus \mathrm{H}_p(B_{1}) \to \mathrm{H}_{p}(|\mathrm{L}^n|) \to \mathrm{H}_{p-1}(A_{1} \cap B_{1}) \to \cdots,
\end{align*}
proves that 
\begin{align*}
\mathrm{H}_{p}(|\mathrm{L}^n|) = 
\left\{\begin{array}{ll}
\mathrm{H}_{p}(A_1) & \text{if } p > 1, \\ 
\mathrm{H}_1(A_1) \oplus \Bbbk & \textit{if } p = 1.
\end{array}\right.
\end{align*}
Thanks to the properties
\begin{itemize}
\item $A_{2} \cup B_{2} = | \mathrm{L}^n |$,
\item $(A_1 \cap A_{2}) \cap (A_1 \cap B_{2}) \simeq \{\star_1,\star_2\}$,
\item $A_1 \cap B_2 \simeq \{\star\}$,
\end{itemize}
the Mayer-Vietoris exact sequence applied to the covering $A_1 = (A_1 \cap A_2) \cup (A_1 \cap B_2)$ proves that 
\begin{align*}
\mathrm{H}_{p}(A_1) = 
\left\{\begin{array}{ll}
\mathrm{H}_{p}(A_1 \cap A_2) & \text{if } p > 1, \\ 
\mathrm{H}_1(A_1 \cap A_2) \oplus \Bbbk & \textit{if } p = 1.
\end{array}\right.
\end{align*}
Applying successively this procedure leads to :
\begin{align*}
\mathrm{H}_{p}(| \mathrm{L}^n |) = 
\left\{\begin{array}{ll}
\mathrm{H}_{p}(A_1 \cap \cdots \cap A_n) & \text{if } p > 1, \\ 
\mathrm{H}_1(A_1 \cap \cdots \cap A_n) \oplus \Bbbk^{n-1} & \textit{if } p = 1.
\end{array}\right.
\end{align*}
Therefore 
\begin{align*}
\mathrm{H}_{p}(| \mathrm{L}^n |) = 
\left\{\begin{array}{ll}
0 & \text{if } p > 1, \\ 
\Bbbk^{n} & \textit{if } p = 1.
\end{array}\right.
\end{align*}
\end{proof}

\subsection{$\mathbf{L}$-homology of cubical sets} Let $\mathrm{X}$ be a cubical set. For all $n \in \mathbb{N}$ let us denote by $\mathrm{QL}_{n}(\mathrm{X})$ the module $\mathrm{Q}_n(\mathrm{LX})$, and by $\mathrm{CL}_{n}(\mathrm{X})$ the quotient of $\mathrm{QL}_{n}(\mathrm{X})$ by the subspace $\mathrm{DL}_n(\mathrm{X}) := \mathrm{D}_n(\mathrm{X}) \cap \mathrm{QL}_n(\mathrm{X})$. These constructions are natural in $\mathrm{X}$ and so induce functors
\begin{align*}
\mathrm{QL}_{n} : \mathbf{cSet} \to \mathbf{\Bbbk Mod} \, \text{ and } \, \mathrm{CL}_{n} : \mathbf{cSet} \to \mathbf{\Bbbk Mod}. 
\end{align*} 
Thanks to cubical identities, the degree $-1$ graded map $d = \sum_{i=1}^n (-1)^{i+1} (d_{i,1}-d_{i,0})$ defines chain complex structures on graded modules $\mathrm{QL}_{\bullet}(\mathrm{X})$ and $\mathrm{CL}_{\bullet}(\mathrm{X})$. These constructions are natural in $\mathrm{X}$ and so induce functors
\begin{align*}
\mathrm{QL}_{\bullet} : \mathbf{cSet} \to \mathbf{Ch}^+ \, \text{ and } \, \mathrm{CL}_{\bullet} : \mathbf{cSet} \to \mathbf{Ch}^+. 
\end{align*} 
\par
The \textit{unormalized $\mathbf{L}$-homology of} $\mathrm{X}$, denoted $\mathrm{HL^{u}_{\bullet}}(\mathrm{X})$, is the homology of the chain complex $\mathrm{QL}_{\bullet}(\mathrm{X})$. The \textit{$\mathbf{L}$-homology of} $\mathrm{X}$, denoted $\mathrm{HL_{\bullet}}(\mathrm{X})$, is the homology of the chain complex $\mathrm{CL}_{\bullet}(\mathrm{X})$.

\subsection{A long exact sequence relating $\mathrm{H}_{\bullet}$ and $\mathrm{HL}_{\bullet}$} The natural transformation $\mathrm{inc : L \to id_{\mathbf{cSet}}}$ induces natural transformations $\mathrm{C_{\bullet}(inc)}$ from $\mathrm{CL}_{\bullet}$ to $\mathrm{C_{\bullet}}$ and $\mathrm{H_{\bullet}(inc)}$ from $\mathrm{HL}_{\bullet}$ to $\mathrm{H_{\bullet}}$ . Then we get a short exact sequence (natural in $\mathrm{X}$) in the category of chain complexes
\begin{align*}
\mathrm{CL_{\bullet}(X)} \hookrightarrow  \mathrm{C_{\bullet}(X)} \twoheadrightarrow \mathrm{C^{rel}_{\bullet}(X)} := \mathrm{C_{\bullet}(X)}/\mathrm{CL_{\bullet}(X)}, 
\end{align*}
which induces a long exact sequence (natural in $\mathrm{X}$) 
\begin{align}
\cdots \to \mathrm{H^{rel}_{n+1}(X)} \to \mathrm{HL_{n}(X)} \to \mathrm{H_{n}(X)} \to \mathrm{H^{rel}_{n}(X)} \to \cdots \label{long exact sequence linking HL and H}
\end{align}

\subsection{A long exact sequence relating $\mathrm{H}_{\bullet}$ and $\mathrm{H}\Gamma_{\bullet}$} The natural transformation $\mathrm{proj : id_{\mathbf{cSet}} \to} \Gamma$ induces natural transformations $\mathrm{C_{\bullet}(proj)}$ from $\mathrm{C}_{\bullet}$ to $\mathrm{C}\Gamma_{\bullet}$ and $\mathrm{H_{\bullet}(proj)}$ from $\mathrm{H}_{\bullet}$ to $\mathrm{H}\Gamma_{\bullet}$ . Then we get a short exact sequence (natural in $\mathrm{X}$) in the category of chain complexes
\begin{align*}
\mathrm{Ker\big(C_{\bullet}(proj)\big)} =: \mathrm{C^{sub}_{\bullet}(X)} \hookrightarrow  \mathrm{C_{\bullet}(X)} \twoheadrightarrow \mathrm{C}\Gamma_{\bullet}(X), 
\end{align*}
which induces a long exact sequence (natural in $\mathrm{X}$) 
\begin{align*}
\cdots \to \mathrm{H^{sub}_{n}(X)} \to \mathrm{H_{n}(X)} \to \mathrm{H}\Gamma_{n}(X) \to \mathrm{H^{sub}_{n-1}(X)} \to \cdots 
\end{align*}

\subsection{A differential graded dendriform coalgebra structure on the Leibniz complex} Let $\mathrm{X}$ be a cubical set. Let us define two degree $0$ maps $\Delta_{\succ}(\mathrm{X})$ and $\Delta_{\prec}(\mathrm{X})$ from $\mathrm{CL}_{\bullet}(\mathrm{X})$ to $\mathrm{CL}_{\bullet}(\mathrm{X}) \otimes \mathrm{CL}_{\bullet}(\mathrm{X})$ by the following formulas :
\begin{align}\displaystyle
\Delta_{\succ}(\mathrm{X)}_n(x) = \bigoplus_{p+q = n} \sum_{\sigma \in \mathrm{Sh}_{p+q}^{1}} \epsilon(\sigma) \, (d_{\sigma(p+1),0}\cdots d_{\sigma(p+q),0})(x) \otimes (d_{1,1} \cdots d_{\sigma(p),1})(x) \label{Deltasucc}
\end{align}
and
\begin{align}\displaystyle
\Delta_{\prec}(\mathrm{X})_n(x) =  \bigoplus_{p+q = n} \sum_{\sigma \in \mathrm{Sh}_{p+q}^{p+1}} \epsilon(\sigma) \, (d_{1,0}\cdots d_{\sigma(p+q),0})(x) \otimes (d_{\sigma(1),1} \cdots d_{\sigma(p),1})(x) \label{Deltaprec} 
\end{align}
for all $x \in \mathrm{X}_n$ and $n \geq 2$. With enough stamina it is possible to check that these formulas define chain complex morphisms. A more conceptual way to prove that such chain morphisms exist and are homotopy unique is to use once again the acyclic models method : Indeed, define $\Delta_{\succ}$ and $\Delta_{\prec}$ in dimensions $2$ by the formulas \eqref{Deltasucc} and \eqref{Deltaprec}. These maps induce natural transformations between homology functors 
\begin{align*}
(\Delta_{\prec})_{2} \, , \, (\Delta_{\succ})_2: \mathrm{HL}_{2}(-,\Bbbk) \to \mathrm{HL}_{1}(-,\Bbbk) \otimes \mathrm{HL}_{1}(-,\Bbbk).
\end{align*}
To define $\Delta_{\succ}$ and $\Delta_{\prec}$ in higher dimensions we use the acyclic models method. By Proposition \ref{Proposition : CLn representable} the functor $\mathrm{CL}_{n}$ is representable by $\mathrm{L}^n$ for all $n \in \mathbb{N}$, and by Proposition \ref{Computation of H(Ln)} the homology groups $\mathrm{H}_{p}\big(\mathrm{CL}_{\bullet}(\mathrm{L}^n) \otimes \mathrm{CL}_{\bullet}(\mathrm{L}^n)\big)$ vanish for all $p > 2$ and $n \in \mathbb{N}$. The acyclic models theorems \ref{Theorem : acyclic models theorem 1} and \ref{Theorem : acyclic models theorem 2} imply that there are homotopy unique natural transformations $\Delta_{\succ}$ and $\Delta_{\prec}$ extending those already defined in dimensions $2$.
\begin{theorem} \label{Theorem : Zinbiel up to homotopy coalgebra structure on Leibniz homology}
Let $X$ be a cubical set. Then $(\mathrm{CL_{\bullet}(X)},\Delta_{\succ},\Delta_{\prec})$ is a homotopy coZinbiel coalgebra (in the category of chain complexes).
\end{theorem}

\begin{proof} 
The maps $(\mathrm{id} \otimes \Delta_{\succ}) \circ \Delta_{\succ}$ and $(\Delta \otimes \mathrm{id}) \circ \Delta_{\succ}$ are two maps from $\mathrm{CL_{\bullet}}$ to $\mathrm{CL_{\bullet} \otimes CL_{\bullet} \otimes CL_{\bullet}}$ which coincide in degree $3$. The acyclic model theorem \ref{Theorem : acyclic models theorem 2} implies that these two maps are homotopic. 
In the same way we prove that $(\mathrm{id} \otimes \Delta_{\prec}) \circ \Delta_{\succ} \simeq(\Delta_{\succ} \otimes \mathrm{id}) \circ \Delta_{\prec}$ and $(\Delta_{\prec} \otimes \mathrm{id}) \circ \Delta_{\prec} \simeq (\mathrm{id} \otimes \Delta) \circ \Delta_{\prec}$. Therefore $(\mathrm{CL_{\bullet}(X)},\Delta_{\succ},\Delta_{\prec})$ is a homotopy codendriform coalgebra.
\vskip 0.2cm
First we prove that there exist a homotopy in dimension $2$ from $\Delta_{\succ}$ to $\tau \circ \Delta_{\prec}$. Then for higher dimensions the result will be a consequence of the acyclic models theorem \ref{Theorem : acyclic models theorem 2}. 
In dimension $2$ we have 
\begin{align*}
\Delta_\succ(\mathrm{X})_2(x) - (\tau \circ \Delta_{\prec})(\mathrm{X})_2(x) &= 
(
 d_{1,0}x \otimes d_{2,1}x) - (
  - d_{1,1}x \otimes d_{2,0}x)
= d_{1,0}x \otimes dx
= (\mathrm{id} \otimes d)(d_{1,0}x \otimes x).
\end{align*}
Then $\mathrm{h} : \mathrm{CL}_{\bullet} \to \big(\mathrm{CL}_{\bullet} \otimes \mathrm{CL}_{\bullet}\big)[1]$ defined by 
$\mathrm{h(X)}_n(x) := 
d_{1,0}x \otimes x $
is a homotopy of dimensions $\leq 2$ from $\Delta_{\succ}$ to $\tau \circ  \Delta_{\prec}$. The acyclic model theorem \ref{Theorem : acyclic models theorem 2} implies that the maps $\Delta_{\succ}$ and $\tau \circ  \Delta_{\prec}$ are homotopic.
As a consequence $(\mathrm{CL_{\bullet}(X)},\Delta_{\succ},\Delta_{\prec})$ is a homotopy coZinbiel coalgebra.
\end{proof}

\begin{corollary} Let $\mathrm{X}$ be a cubical set. Then $(\mathrm{HL}_{\bullet}(X,\Bbbk),\Delta)$ is a connected  coZinbiel coalgebra (in the category of graded vector space). 
\end{corollary}

\section{$\mathbf{L}$-homology for groups and rack homology}
In this section we compute explicitely the $\mathbf{L}$-set associated to the cubical nerve of a group $G$. We proved that the associated Leibniz chain complex is exactly the chain complex which computes the \textit{rack homology} of $G$. We deduce from this result and the previous section the existence of a coZinbiel coalgebra structure on the rack homology of a group $G$ and the existence of a long exact sequence relating group homology and rack homology. More generally we prove that rack homology is provided with a coZinbiel coalgebra structure.

\subsection{Cubical and simplicial nerves of a category} Let $\mathbf{C}$ be a category. The \textit{cubical nerve} of $\mathbf{C}$ is the cubical set $\mathrm{N^{\square}\mathbf{C}} := \mathrm{Hom}_{\mathbf{Cat}}\big(-,\mathbf{C}\big) : \square^{\mathrm{op}} \to \mathbf{Set}$, and the \textit{simplicial nerve} of $\mathbf{C}$ is the simplicial set $\mathrm{N^{\Delta}\mathbf{C}} := \mathrm{Hom}_{\mathbf{Cat}}\big(-,\mathbf{C}\big) : \Delta^{\mathrm{op}} \to \mathbf{Set}$. Using these functorial constructions we define two functors from $\mathbf{Cat}$ to $\mathbf{Ch}^+$ : $\mathrm{C}_{\bullet} \circ \mathrm{N}^{\square}$ and $\mathrm{C}_{\bullet} \circ \mathrm{N}^{\Delta}$.
\begin{align*}
\xymatrix{
\mathbf{Cat} \ar[r]^{\mathrm{N}^{\square}} \ar[d]_{\mathrm{N}^{\Delta}} & \mathbf{cSet} \ar[d]^{\mathrm{C}_{\bullet}}
\\
\mathbf{sSet} \ar[r]_{\mathrm{C}_{\bullet}} & \mathbf{Ch}^+
}
\end{align*}
\begin{theorem} \label{equivalence between cubical and simplicial nerves}
Let $\mathbf{C}$ be a category. The chain complexes $\mathrm{C_{\bullet}(N^{\square}\mathbf{C})}$ and $\mathrm{C_{\bullet}(N^{\Delta}\mathbf{C})}$ are quasi isomorphic.
\end{theorem}

\begin{proof}

Define $\mathrm{S}$ in dimension $0$ by  $\mathrm{S}_0 := \mathrm{id} : \mathrm{C}_0(\mathrm{N}^{\square}\mathbf{C}) \to \mathrm{C}_0(\mathrm{N}^{\Delta}\mathbf{C})$. By induction we extend this map to a morphism of chain complexes which is unique up to homotopy using acyclic models theorems. Indeed, by Lemma \ref{Proposition : Cn representable} the functor $\mathrm{C}_{n}$ is representable by $\square^{n}$ for all $n \in \mathbb{N}$, and by Lemma \ref{lemma : deltan is contractible} the homology groups $\mathrm{H}_{n-1}\big(\mathrm{C}_{\bullet}(\Delta^{n})\big)$ vanish for all $n \geq 2$. The acyclic models theorems \ref{Theorem : acyclic models theorem 1} and \ref{Theorem : acyclic models theorem 2} imply the existence of a homotopy unique morphism of chain complexes $\mathrm{S} : \mathrm{C_{\bullet}(N^{\square}\mathbf{C})} \to \mathrm{C_{\bullet}(N^{\Delta}\mathbf{C})}$ extending $\mathrm{S}_0$.

\vskip 0.2cm

In the same way Lemma \ref{lemma : Cndelta is representable} and Lemma \ref{lemma : squaren is contractible} imply that there exists a homotopy unique morphism of chain complexes $\mathrm{S}^{-1} : \mathrm{C_{\bullet}(N^{\Delta}\mathbf{C})} \to \mathrm{C_{\bullet}(N^{\square}\mathbf{C})}$ extending $\mathrm{S_0}^{-1} := \mathrm{id} : \mathrm{C}_0(\mathrm{N}^{\Delta}\mathbf{C}) \to \mathrm{C}_0(\mathrm{N}^{\square}\mathbf{C})$.

\vskip0.2cm

The chain map $\mathrm{S} \circ \mathrm{S}^{-1}$ coincides with $\mathrm{id} : \mathrm{C_{\bullet}(N^{\Delta}\mathbf{C})} \to \mathrm{C_{\bullet}(N^{\Delta}\mathbf{C})}$ in dimension $0$. The acyclic models Theorem \ref{Theorem : acyclic models theorem 2} implies that these two chain maps are homotopic. In the same way $\mathrm{S}^{-1} \circ \mathrm{S}$ is homotopic to $\mathrm{id} : \mathrm{C_{\bullet}(N^{\square}\mathbf{C})} \to \mathrm{C_{\bullet}(N^{\square}\mathbf{C})}$.  
\end{proof}

An explicit formula for $\mathrm{S} = \{\mathrm{S}_n\}_{n \in \mathbb{N}}$ is given by 
\begin{align*}
\mathrm{S}_n := \sum_{\sigma \in \mathbb{S}_n} \epsilon(\sigma) \, \sigma^* 
\end{align*}
where $\sigma$ is the functor from $\Delta_n$ to $\square_n$ defined by $\sigma(i) := \{\sigma(1),\dots,\sigma(i)\}$. 
\subsection{Cubical and simplicial nerves of a group} Let $G$ be a group. We still denote by $G$ the groupoid with set of objects $\{e\}$ and set of morphisms $G$. The composition is defined using the group multiplication:
\begin{align*}
e \stackrel{g}{\longrightarrow} e \stackrel{h}{\longrightarrow} e := e \stackrel{gh}{\longrightarrow} e
\end{align*}
The identity is given by the neutral element in $G$ : 
\begin{align*}
\mathrm{id}_e := e \stackrel{e}{\longrightarrow} e
\end{align*}
The \textit{cubical and simplicial nerves of $G$} are respectively the cubical and simplicial sets $\mathrm{N^{\square}}G$ and $\mathrm{N^{\Delta}}G$ of the groupoid $G$. Let us denote by $\mathrm{C^{\square}_{\bullet}}(G,\Bbbk)$ and $\mathrm{C^{\Delta}_{\bullet}}(G,\Bbbk)$ the chain complexes respectively associated to the cubical and simplicial nerves of $G$. By Theorem \ref{equivalence between cubical and simplicial nerves} these two chain complexes are quasi isomorphic, and so there is an isomorphism of their associated homologies 
\begin{align*}
\mathrm{H_{\bullet}^{\square}}(G,\Bbbk) \simeq \mathrm{H_{\bullet}^{\Delta}}(G,\Bbbk).
\end{align*}
A direct computation proves that the simplicial nerve $\mathrm{N}^{\Delta}G$ of the groupoid $G$ is isomorphic to the nerve $\mathrm{N}G$ of the group $G$ defined by
\begin{align*}
\mathrm{N}G &:= G^n,\\
\mathrm{N}G (\delta_i)(g_1,\dots,g_n) &:= 
\left\{\begin{array}{ll}
(g_2,\dots,g_n) & \text{ if } i = 0, \\
(g_1,\dots,g_ig_{i+1},\dots,g_n) & \text{ if } 1 \leq i \leq n,\\
(g_1,\dots,g_{n-1}) & \text{ if } i = n, \\
\end{array}\right. \\
\mathrm{N}G(\sigma_{i})(g_1, \dots,g_n) &:= (g_1, \dots, g_{i-1},e,g_i,\dots,g_n).
\end{align*} 
Therefore there are isomorphisms between cubical homology, simplicial homology and Eilenberg-MacLane homology.
\begin{align*}
\mathrm{H_{\bullet}^{\square}}(G,\Bbbk) \simeq \mathrm{H_{\bullet}^{\Delta}}(G,\Bbbk) \simeq \mathrm{H}_{\bullet}(G,\Bbbk).
\end{align*}
\subsection{Computation of the $\mathbf{L}$-nerve $\mathrm{L\big(N^{\square}G\big)}$ of a group $G$} Let $G$ be a group. Let $\mathrm{N^R}G$ be the cubical set defined by 

\begin{align*}
\mathrm{N^R}G_{n} &:= G^n ,\\
\mathrm{N^R}G(\delta_{i,\epsilon})(g_1, \dots,g_n) &:= 
\left\{\begin{array}{ll}
(g_1,\dots,g_{i-1},g_{i+1},\dots,g_n) & \text{ if } \epsilon = 0, \\
(g_i^{-1}g_1g_i,\dots,g_i^{-1}g_{i-1}g_i,g_{i+1},\dots,g_n) & \text{ if } \epsilon = 1.
\end{array}\right. \\
\mathrm{N^R}G(\sigma_{i})(g_1, \dots,g_n) &:= (g_1, \dots, g_{i-1},e,g_i,\dots,g_n).
\end{align*}
\begin{proposition} \label{Theorem : bijection from L(N(G,M)) to MGn} The cubical sets $\mathrm{L\big(N^{\square}G\big)}$ and $\mathrm{N^R}G$ are isomorphic. 
\end{proposition}
\begin{proof}
First, let us compute $\mathrm{L(N_{\square}}G)_n$ for all $n \in \mathbb{N}$. Let $n \in \mathbb{N}$ and $\mathrm{F} \in \mathrm{L(N_{\square}}G)_n$ be fixed. Let us prove that $\mathrm{F}$ is completely and uniquely determined by its image on the morphisms $\emptyset \to \{i\}$ for all $1 \leq i \leq n$. In other words let us prove that the map
\begin{align} \label{bijection from L(N(G,M)) to MGn}
\mathrm{F} \mapsto \big(\mathrm{F}(\emptyset \to \{1\}),\mathrm{F}(\emptyset \to \{2\}),\dots,\mathrm{F}(\emptyset \to \{n\})\big).
\end{align}
is a bijection from $\mathrm{L\big(N_{\square}G\big)}_n$ to $G^n$.
\vskip 0.2cm
By definition $\mathrm{F}$ is a functor from $\square_n$ to $G$, so
\begin{align*}
\mathrm{F}(A \to B) &= \mathrm{F}(A \to A \sqcup \{i_1\}) \cdots \mathrm{F}(A \sqcup \{i_1,\dots, i_{k-1}\} \to B)
\end{align*}
for all morphisms $A \to B$ in $\square_n$ with $B = A \sqcup \{i_1, \dots, i_k\}$ where $i_1 < \cdots < i_k$. Then $\mathrm{F}$ is determined by its images on morphisms of the form $A \to A \sqcup \{i\}$ with $i > \max(A)$. Moreover
\begin{align*}
A \to A \sqcup \{i\} = \delta_A(\delta_{2,0})^{n-i+2}(\emptyset \to \{1\}) \text{ and } \emptyset \to \{i\} = \delta_{\emptyset}(\delta_{2,0})^{n-i+2}(\emptyset \to \{1\}),
\end{align*} 
then
\begin{align*}
\mathrm{F}(A \to A \sqcup \{i\}) = (d_{2,0})^{n-i+2}d_A \mathrm{F}(\emptyset \to \{1\}) = (d_{2,0})^{n-i+2}d_\emptyset \mathrm{F}(\emptyset \to \{1\}) = \mathrm{F}(\emptyset \to \{i\}).
\end{align*}
Thus the map \eqref{bijection from L(N(G,M)) to MGn} is a bijection between $\mathrm{L(N_{\square}}G)_n$ and $G^n = \mathrm{N^R}G_n$.
\vskip 0.2cm
Now let us compute the maps $d_{i,\epsilon} : \mathrm{L(N_{\square}}G)_n \to \mathrm{L(N_{\square}}G)_{n-1}$ for all $1 \leq i \leq n$ and $\epsilon \in \{0,1\}$ under the bijection \eqref{bijection from L(N(G,M)) to MGn}. Let $\mathrm{F} \in \mathrm{L(N_{\square}}G)_n$ and $(g_1,\dots,g_n)$ be its corresponding element in $G^n$ under the bijection \eqref{bijection from L(N(G,M)) to MGn}. We have 
\begin{align*}
(d_{i,\epsilon}\mathrm{F})(\emptyset \to \{j\}) &=
\left\{
\begin{array}{ll}
\mathrm{F}(\emptyset \to \{j\}) = g_j& \text{ if } i > j , \epsilon = 0,\\
\mathrm{F}(\emptyset \to \{j+1\}) =g_{j+1} & \text{ if } i \leq j , \epsilon = 0,\\
\mathrm{F}(\{i\} \to \{j,i\}) = \mathrm{F}(\{i\} \to \emptyset \to \{j\} \to \{j,i\}) = g_{i}^{-1}g_jg_i & \text{ if } i > j , \epsilon = 1,\\
\mathrm{F}(\{i\} \to \{i,j+1\}) = \mathrm{F}(\emptyset \to \{j+1\}) = g_{j+1} & \text{ if } i \leq j , \epsilon = 1,
\end{array}
\right. 
\end{align*}

\vskip0.2cm

Finally let us compute the maps $s_i : \mathrm{L(N_{\square}}G)_n \to \mathrm{L(N_{\square}}G)_{n+1}$ for all $1 \leq i \leq n+1$ under the bijection \eqref{bijection from L(N(G,M)) to MGn}. Let $\mathrm{F} \in \mathrm{L(N_{\square}}G)_n$ and $(g_1,\dots,g_n)$ be its corresponding element in $G^n$ under the bijection \eqref{bijection from L(N(G,M)) to MGn}. We have
\begin{align*}
(s_i\mathrm{F})(\emptyset \to \{i\}) &= 
\left\{
\begin{array}{ll}
\mathrm{F}(\emptyset \to \{j\}) = g_j& \text{ if } i > j , \\
\mathrm{F}(\mathrm{id}_{\emptyset}) = e & \text{ if } i = j , \\
\mathrm{F}(\emptyset \to \{j-1\}) = g_{j-1} & \text{ if } i < j.\\
\end{array}
\right. 
\end{align*}

\vskip0.2cm

Thus \eqref{bijection from L(N(G,M)) to MGn} induces an isomorphism between the cubical sets $\mathrm{L\big(N^{\square}G\big)}$ and $\mathrm{N^R}G$ 
\end{proof}

A consequence of \ref{Theorem : Zinbiel up to homotopy coalgebra structure on Leibniz homology} is that the homology of the cubical set $\mathrm{N^R}(G)$ is provided with a coZinbiel coalgebra structure. This result is consistent with the conjecture of J.-L. Loday about the existence of a Leibniz homology defined for groups provided with a coZinbiel coalgebra structure.
\begin{align}
\mathrm{HL}_{\bullet}(-,\Bbbk) := \mathrm{H}_{\bullet} \circ \mathrm{C}_{\bullet}(-,\Bbbk) \circ \Bbbk.\mathrm{N^R} : \mathbf{Grp} \to \mathbf{Zinb}^c \label{HL for groups}
\end{align}
In the following part we prove that this $\mathbf{L}$-homology defined on the category of groups can be extended to a larger category called the \textit{category of racks}. On this category of racks the $\mathbf{L}$-homology functor becomes the usual rack homology functor.
 
\subsection{Rack homology}
Let $G$ be a group. Theorem \ref{Theorem : bijection from L(N(G,M)) to MGn} emphasizes that the maps $d_{i,\epsilon}$ and $s_i$ are defined using only the \textit{conjugation} in $G$ and not the product in $G$. Moreover we can check that the cubical relations are consequences of the following identities :
\begin{align}
e \lhd g &= e, \label{neutral element 1}\\
g \lhd e &= g, \label{neutral element 2}\\ 
(k \lhd h) \lhd g &= (k \lhd g) \lhd (h \lhd g), \label{rack identity}
\end{align}
where $h \lhd g = g^{-1}hg$. As a consequence this cubical set can be defined in general for any set $X$ provided with an operation $\lhd : X \times X \to X$ satisfying these identities. This remark leads to the notion of (pointed) rack.
\begin{definition}
A \textnormal{rack} is a set $X$ provided with a product $\lhd : X \times X \to X$ satisfying the \textit{rack identity} \eqref{rack identity} for all $g,h,k \in X$, and such that the map $- \lhd g : X \to X$ is a bijection for all $g \in X$. A \textnormal{pointed rack} is a rack $(X,\lhd)$ provided with an element $e \in X$, called \textnormal{neutral element}, satisfying equations \eqref{neutral element 1} and \eqref{neutral element 2} for all $g \in X$.
\end{definition} 
In the sequel all racks
will be considered pointed. Given a rack $X$
there is a cubical set $\mathrm{N^R}X$, called the \textit{nerve of $X$}, and defined by the following formulas.
\begin{align*}
\mathrm{N^R}X_{n} &:= X^n ,\\
\mathrm{N^R}X(\delta_{i,\epsilon})(x_1, \dots,x_n) &:= 
\left\{\begin{array}{ll}
(x_1,\dots,x_{i-1},x_{i+1},\dots,x_n) & \text{ if } \epsilon = 0, \\
(x_1 \lhd x_i,\dots,x_{i-1} \lhd x_i,x_{i+1},\dots,x_n) & \text{ if } \epsilon = 1.
\end{array}\right. \\
\mathrm{N^R}X(\sigma_{i})(x_1, \dots,x_n) &:= (x_1, \dots, x_{i-1},e,x_i,\dots,x_n).
\end{align*} 
The chain complex associated to this cubical set is called the \textit{rack chain complex} of $X$ and is denoted by $\mathrm{CR_{\bullet}}(X,\Bbbk)$. Its homology is called the \textit{rack homology} of $X$ and is denoted by $\mathrm{HR}_{\bullet}(X,\Bbbk)$.
\vskip 0.2cm
The nerve of a rack $X$ is a $\mathbf{L}$-set, therefore the rack chain complex of $X$ is provided with a homotopy coZinbiel coalgebra structure. 
\begin{theorem} \label{Theorem : Zinbiel up to homotopy coalgebra structure on rack homology} Let $X$ be a rack. The rack homology of $X$ is provided with a coZinbiel coalgebra structure.
\begin{align*}
\mathrm{HR}_{\bullet}(-,\Bbbk) : \mathbf{Rack} \to \mathbf{Zinb}^c
\end{align*} 
\end{theorem}
Given $n \in \mathbb{N}$ an explicit formula for $(\Delta_{\prec})_n : \mathrm{HR}_n(X,\Bbbk) \to \oplus_{p+q = n} \mathrm{HR}_p(X,\Bbbk) \otimes \mathrm{HR}_q(X,\Bbbk)$ is :
\begin{align*}
(\Delta_{\prec})_n[x_1,\dots,x_n] = \sum_{p+q=n}\sum_{\sigma \in \mathrm{Sh}_{p,q}^1} \epsilon(\sigma) \, [x_1,x_{\sigma(2)},\dots,x_{\sigma(p)}] \otimes [x_{p+1}^{\sigma},\dots,x_{p+q}^{\sigma}]
\end{align*}
where $x_i^{\sigma} := x_{\sigma(i)} \lhd x_{i_1} \lhd \cdots \lhd x_{i_k}$ with $i_{j} \in \{\sigma(1),\dots,\sigma(p)\}$ and $i_k > \cdots > i_1 > \sigma(i)$.
\vskip 0.2cm
The functor $\mathrm{Conj}:\mathbf{Grp} \to \mathbf{Rack}$ from the category of groups to the category of racks makes the following diagram commutative
$$
\xymatrix{
\mathbf{Rack} \ar[rr]^{ \mathrm{HR}_{\bullet}(-,\Bbbk)} & & \mathbf{Zinb}^c 
\\
\mathbf{Grp} \ar[u]^{\mathrm{Conj}} \ar[rru]_{\quad \mathrm{HL}_{\bullet}(-,\Bbbk)} & &
}
$$
where $\mathrm{HL}_{\bullet}(-,\Bbbk)$ is defined by \eqref{HL for groups}. This result is consistent with the conjecture of J.-L. Loday about the existence of mathematical objects (coquecigrues) whose groups naturally carry the structure and whose natural homology theory is provided with a coZinbiel coalgebra structure.
\vskip 0.2cm
The long exact sequence \eqref{long exact sequence linking HL and H} applied to the nerve of a group induces the following theorem relating rack homology and group homology.

\begin{theorem}\label{Theorem : long exact sequence linking HR and H}
Let $G$ be a group. There is a long exact sequence 
\begin{align*}
\cdots \to \mathrm{H}^{\mathrm{rel}}_{n+1}(G,\Bbbk) \to \mathrm{HR}_{n}(G,\Bbbk) \stackrel{\mathrm{S}_n}{\to} \mathrm{H}_{n}(G,\Bbbk) \to \mathrm{H}^{\mathrm{rel}}_{n}(G,\Bbbk) \to \mathrm{HR}_{n-1}(G,\Bbbk)  \stackrel{\mathrm{S}_{n-1}}{\to} \mathrm{H}_{n-1}(G,\Bbbk) \to \cdots
\end{align*}
\end{theorem}
Given $n \in \mathbb{N}$ an explicit formula for $\mathrm{S}_n : \mathrm{HR}_n(G,\Bbbk) \to \mathrm{H}_{n}(G,\Bbbk)$ is:
\begin{align*}
\mathrm{S}_n(g_1,\dots,g_n) = \sum_{\sigma \in \mathbb{S}_n} \epsilon(\sigma) \, (g^{\sigma}_1,\dots,g^{\sigma}_n)
\end{align*}
where $g_i^{\sigma} := g_{\sigma(i)} \lhd g_{i_1} \lhd \cdots \lhd g_{i_k}$ with $i_j \in \{\sigma(1),\dots,\sigma(i-1)\}$ and $i_k > i_{k-1} > \cdots > i_1 > \sigma(i)$.
\vskip 0.2cm
This result is consistent with the conjecture of J.-L. Loday about the existence of a natural morphism of cocommutative coalgebras from Leibniz homology to the usual group homology.

\section{Rack homology of abelian groups} \label{section : rack homology of abelian groups}
In this section $G$ is an abelian group and $\mu : G \times G \to G$ is the commutative multiplication in $G$. Group homology of abelian groups with coefficients in a field of characteristic $0$ is well known. Using the Pontryagin product we can prove for all $n \in \mathbb{N}$ the following isomorphism (cf. \cite{Brown_Cohomology_of_groups} pp.121).
\begin{align}
\mathrm{H}_{n}(G,\Bbbk) \simeq \Lambda^n(G \otimes \Bbbk) \label{group homology of an abelian group}
\end{align}
There is a similar result for the rack homology of an abelian group. This is the content of the following theorem.
\begin{theorem} Let $G$ be an abelian group. For all $n \in \mathbb{N}$ there is an isomorphism
\begin{align}
\mathrm{HR}_{n}(G,\Bbbk) \simeq \mathrm{T}^n(\Bbbk[G \setminus \{0\}]). \label{rack homology of an abelian group}
\end{align}
If  $\Bbbk$ is a field of characteristic $0$, then under bijections \eqref{group homology of an abelian group} and \eqref{rack homology of an abelian group} the map $\mathrm{S}_n : \mathrm{HR}_n(G,\Bbbk) \to \mathrm{H}_n(G,\Bbbk)$ is the canonical projection $\mathrm{T}^n(\Bbbk[G \setminus \{0\}]) \twoheadrightarrow \Lambda^n(G \otimes \Bbbk)$ for all $n \in \mathbb{N}$.
\end{theorem}

\begin{proof} The group $G$ being abelian the conjugation is trivial. It implies that the differential of the chain complex $\mathrm{CR}_{\bullet}(G,\Bbbk)$ is equal to zero. Therefore the homology groups $\mathrm{HR}_{n}(G,\Bbbk)$ are isomorphic to $\mathrm{CR}_n(G,\Bbbk) = \Bbbk[G^n] \simeq \Bbbk[G]^{\otimes n}$ for all $n \in \mathbb{N}$.
\vskip 0.2cm
The conjugation in $G$ being trivial the map $\mathrm{S}_n:\mathrm{HR}_n(G,\Bbbk) \to \mathrm{H}_{n}(G,\Bbbk)$ is equal to $\mathrm{S}_n(g_1,\dots,g_n) = \sum_{\sigma \in \mathbb{S}_n} \epsilon(\sigma) \,  (g_{\sigma(1)},\dots,g_{\sigma(n)})$ for all $n \in \mathbb{N}$. Under bijections \eqref{group homology of an abelian group} and \eqref{rack homology of an abelian group} this map is the canonical projection.
\end{proof}
\subsection{A commutative Hopf algebra structure on the group homology of an abelian group} The group $G$ being abelian the multiplication $\mu$ is a group morphism. As a consequence the chain complex $\mathrm{C}_{\bullet}(G,\Bbbk)$ computing the group homology of $G$ is provided with a product $\star$, called \textit{Pontryagin product}, and defined by the formula
\begin{align*}
\star : \mathrm{C}_{\bullet}(G,\Bbbk) \otimes \mathrm{C}_{\bullet}(G,\Bbbk) \stackrel{\mathrm{EZ}}{\longrightarrow} \mathrm{C}_{\bullet}(G \times G,\Bbbk) \stackrel{\mathrm{C}_{\bullet}(\mu)}{\longrightarrow} \mathrm{C}_{\bullet}(G,\Bbbk). 
\end{align*}
where $\mathrm{EZ}$ is the \textit{Eilenberg-Zilber map}, an inverse to the Alexander-Whitney map. An explicit formula for $\star$ is given by
\begin{align*}
\mathrm{F} \star \mathrm{F}' := \sum_{\sigma \in \mathrm{Sh}_{p,q}} \epsilon(\sigma) \, \mu \circ (\mathrm{F} \times \mathrm{F}') \circ \sigma
\end{align*}
for all $\mathrm{F} \in \mathrm{C}_p(G,\Bbbk), \, \mathrm{F}' \in \mathrm{C}_q(G,\Bbbk)$. In this formula $\sigma$ is the functor from $\Delta_{p+q}$ to $\Delta_p \times \Delta_q$ defined by 
\begin{align*}
\sigma(j) &:= \left\{\begin{array}{ll} 
\big(\sigma^{-1}(j),j-\sigma^{-1}(j)\big) & \text{if } 1 \leq \sigma^{-1}(j) \leq p ,\\
(j-\sigma^{-1}(j)+p,\sigma^{-1}(j)-p) & \text{if } p+1 \leq \sigma^{-1}(j) \leq p+q. 
\end{array} \right.
\end{align*}
Under the bijection $\mathrm{C}_{n}(G,\Bbbk) \simeq \Bbbk.G^n$ the product $\star$ is equal to
\begin{align*}\displaystyle
(g_1,\dots,g_p) \star (g_{p+1}, \dots , g_{p+q}) = \sum_{\sigma \in \mathrm{Sh}_{p,q}} \epsilon(\sigma) \, (g_{\sigma^{-1}(1)},\dots,g_{\sigma^{-1}(p+q)}).
\end{align*}
This product is associative and commutative and thus provides $\mathrm{C}_{\bullet}(G,\Bbbk)$ with a commutative algebra structure. Previously we have seen that $\mathrm{C}_{\bullet}(G,\Bbbk)$ is provided with a cocommutative up to homotopy coalgebra structure (Theorem \ref{Theorem : commutative up to homotopy coalgebra structure on C(X)}). These two structures are compatible, \textit{i.e.} they satisy the \textit{Hopf relation}
\begin{align*}
\Delta \circ \star = \star_{\otimes} \circ (\Delta \otimes \Delta). 
\end{align*}
Therefore the group homology of an abelian group $G$ with coefficients in a field $\Bbbk$ is a commutative Hopf algebra. This Hopf algebra is connected so if $\Bbbk$ is a field of characteristic $0$, then the Hopf-Borel Theorem (\ref{Theorem : Hopf-Borel}) implies that the Hopf algebra $(\mathrm{H}_{\bullet}(G,\Bbbk),\Delta,\star)$ is free and cofree over its primitive part.

\subsection{A coZinbiel-associative bialgebra structure on the rack homology of an abelian group} The multiplication $\mu : G \times G \to G$ being a group morphism it is a rack morphism. As a consequence the chain complex $\mathrm{CR}_{\bullet}(G,\Bbbk)$ computing the rack homology of $G$ is provided with a product $\star$, still called the Pontryagin product, and defined by the same formula as before.
\begin{align*}
\star : \mathrm{CR}_{\bullet}(G,\Bbbk) \otimes \mathrm{CR}_{\bullet}(G,\Bbbk) \stackrel{\mathrm{EZ}}{\longrightarrow} \mathrm{CR}_{\bullet}(G \times G,\Bbbk) \stackrel{\mathrm{CR}_{\bullet}(\mu)}{\longrightarrow} \mathrm{CR}_{\bullet}(G,\Bbbk). 
\end{align*}
An explicit formula for $\star$ on $\mathrm{CR}_{\bullet}(G,\Bbbk)$ is given by :
\begin{align*}
\mathrm{F} \star \mathrm{F}' = \mu \circ (\mathrm{F} \times \mathrm{F}') \circ \mathrm{i}_{p,q}
\end{align*}
where $i_{p,q}$ is the functor from $\square_{p+q}$ to $\square_{p} \times \square_q$ defined by $\mathrm{i}_{p,q}(\epsilon_1,\dots,\epsilon_{p+q}) = \big((\epsilon_1,\dots,\epsilon_p),(\epsilon_{p+1},\dots,\epsilon_{p+q})\big)$. Under the bijection $\mathrm{CR}_{\bullet}(G,\Bbbk) \simeq \Bbbk.G^n$ the product $\star$ is equal to 
\begin{align*}
(g_1,\dots,g_p) \star (g_{p+1},\dots,g_{p+q}) = (g_1,\dots,g_{p+q}).
\end{align*}
This product is associative and thus provides $\mathrm{CR}_{\bullet}(G,\Bbbk)$ with an associative algebra structure. Previously we have seen that $\mathrm{CR}_{\bullet}(G,\Bbbk)$ is provided with a coZinbiel up to homotopy coalgebra structure (Theorem \ref{Theorem : Zinbiel up to homotopy coalgebra structure on Leibniz homology}). These two structures are compatible, \textit{i.e.} they satisfy the \textit{semi-Hopf relation}
\begin{align*}
\Delta_{\prec} \circ \star = \star_{\otimes} \circ (\Delta_{\prec} \otimes \Delta).
\end{align*}
Therefore the rack homology of an abelian group $G$ with coefficients in a field $\Bbbk$ is a coZinbiel-associative bialgebra. This bialgebra  is connected so Theorem \ref{Theorem : Structure theorem for coZinbiel-associative bialgebras} implies that the coZinbiel-associative bialgebra $(\mathrm{HR}_{\bullet}(G,\Bbbk),\Delta_{\prec},\star)$ is free and cofree over its primitive part.

\section{Rack homology of the linear group} Let $R$ be a ring with unit and $\mathrm{GL}_n(R)$ be the group of invertible $n \times n$ matrices. Let $\oplus$ be the associative product defined on the graded group $\{\mathrm{GL}_n(R)\}_{n \in \mathbb{N}}$ by the formula :
\begin{align*}
A \oplus B := \begin{bmatrix} A & 0 \\ 0 & B \end{bmatrix}.
\end{align*}
The \textit{linear group with coefficients in $R$} $\mathrm{GL}(R)$ is the inductive limit of the system $\{i_{m,n} : \mathrm{GL}_{m}(R) \hookrightarrow \mathrm{GL}_{n}(R)\}_{m,n \in \mathbb{N}}$ where $i_{m,n}(A) = A \oplus \mathrm{I}_{n-m}$ ($\mathrm{I}_p$ is the identity matrix on $R^p$). 
\vskip 0.2cm

For all $n \in \mathbb{N}^*$ let $\mu_n : \mathrm{GL}_n(R) \times \mathrm{GL}_n(R) \to \mathrm{GL}_{2n}(R)$ be the group morphism defined by 
\begin{align*}
\mu_n(A,B) := \left\{\begin{array}{ll}
0 & \text{if } i \neq j \, \mathrm{mod} \, 2,\\
a_{\frac{i+1}{2},\frac{j+1}{2}} & \text{if } i = j = 1 \, \mathrm{mod} \, 2,\\
b_{\frac{i}{2},\frac{j}{2}} & \text{if } i = j = 0 \, \mathrm{mod} \, 2.
\end{array}\right.
\end{align*}
The family of maps $\{\mu_n\}_{n \in \mathbb{N}^*}$ induces a group morphism $\mu : \mathrm{GL}(R) \times \mathrm{GL}(R) \to \mathrm{GL}(R)$. The class of the identity matrices in $\mathrm{GL}(R)$ will be denoted by $e$.
\vskip 0.2cm
A relation between $\mu_n$ and $\oplus$ is given by the following lemma. 
\begin{lemma}\label{Lemma : relation between mu and oplus}
For all $n \in \mathbb{N}^*$ there exists $C_n \in \mathrm{GL}_{2n}(R)$ such that for all $A,B \in \mathrm{GL}_{n}(R)$  
\begin{align*}
C_n^{-1}\mu_n(A,B)C_n = A \oplus B.
\end{align*}
\end{lemma}
\begin{proof}
Take $\displaystyle C_n = \prod_{1 \leq i \leq 2n-1} C_{i,i+1}$ where $C_{i,i+1}$ is the matrix exchanging columns $i$ and $i+1$.
\end{proof}
\begin{lemma}\label{Lemma : commutativity of oplus}
For all $m,n \in \mathbb{N}^*$ there exists $D_{m,n} \in \mathrm{GL}_{m+n}(R)$ such that for all $A\in \mathrm{GL}_{m}(R),B \in \mathrm{GL}_n(R)$  
\begin{align*}
D_{m,n}^{-1}(A \oplus B)D_{m,n} = B \oplus A.
\end{align*}
\end{lemma}
\begin{proof}
Take $D_{m,n} := \begin{bmatrix} 0 & \mathrm{I}_n \\ \mathrm{I}_m & 0 \end{bmatrix}$. 
\end{proof}

\subsection{A commutative Hopf algebra structure on the group homology of the linear group}
In this section we recall how to define a graded Hopf algebra structure on the group homology of the linear group $\mathrm{GL}(R)$. Here we use a direct method but one can look at \cite{LodayCyclic} or \cite{Rosenberg_Algebraic_K_Theory} for a topological construction using the $H$-space structure of $\mathrm{BGL}(R)^+$ .
\vskip 0.2cm
The group morphism $\mu$ induces a product $\star$ on $\mathrm{C}_{\bullet}(\mathrm{GL}(R),\Bbbk)$, called the $\textit{Pontryagin product}$, and defined by the formula
\begin{align*}
\star : \mathrm{C}_{\bullet}(\mathrm{GL}(R),\Bbbk) \otimes \mathrm{C}_{\bullet}(\mathrm{GL}(R),\Bbbk) \stackrel{\mathrm{EZ}}{\longrightarrow} \mathrm{C}_{\bullet}(\mathrm{GL}(R) \times \mathrm{GL}(R),\Bbbk) \stackrel{\mathrm{C}_{\bullet}(\mu)}{\longrightarrow} \mathrm{C}_{\bullet}(\mathrm{GL}(R),\Bbbk).
\end{align*}
An explicit formula for $\star$ is given by
\begin{align*}
\mathrm{F} \star \mathrm{F}' := \sum_{\sigma \in \mathrm{Sh}_{p,q}} \epsilon(\sigma) \, \mu \circ (\mathrm{F} \times \mathrm{F}') \circ \sigma
\end{align*}
for all $\mathrm{F} \in \mathrm{C}_p(G,\Bbbk), \, \mathrm{F}' \in \mathrm{C}_q(G,\Bbbk)$. In this formula $\sigma$ is the functor from $\Delta_{p+q}$ to $\Delta_p \times \Delta_q$ defined as before by 
\begin{align*}
\sigma(j) &:= \left\{\begin{array}{ll} 
\big(\sigma^{-1}(j),j-\sigma^{-1}(j)\big) & \text{if } 1 \leq \sigma^{-1}(j) \leq p ,\\
(j-\sigma^{-1}(j)+p,\sigma^{-1}(j)-p) & \text{if } p+1 \leq \sigma^{-1}(j) \leq p+q. 
\end{array} \right.
\end{align*}
Under the bijection $\mathrm{C}_{n}(\mathrm{GL}(R),\Bbbk) \simeq \Bbbk.\mathrm{GL}(R)^n$ the product $\star$ is equal to
\begin{align*}\displaystyle
(g_1,\dots,g_p) \star (g_{p+1}, \dots , g_{p+q}) = \sum_{\sigma \in \mathrm{Sh}_{p,q}} \epsilon(\sigma) \, (g^{\mu}_{\sigma^{-1}(1)},\dots,g^{\mu}_{\sigma^{-1}(p+q)}),
\end{align*}
where $g_i^{\mu} = \left\{\begin{array}{ll} 
\mu(g_i,e) & \text{if } 1 \leq i \leq p\\
\mu(e,g_i) & \text{if } p+1 \leq i \leq p+q
\end{array}\right.$.
\begin{proposition} \label{Proposition : Pontryagin product is associative and commutative in group homology}
The product $\star$ on $\mathrm{H}_{\bullet}(\mathrm{GL}(R),\Bbbk)$ is associative and commutative.
\end{proposition}
One topological way to prove this proposition is to use the topological space $\mathrm{BGL}(R)^+$. The product $\mu$ endows the space $\mathrm{BGL}(R)^+$ with an associative and commutative $H$-space structure, and, using the isomorphism between the singular homology of $\mathrm{BGL}(R)^+$ and the group homology of $\mathrm{GL}(R)$, we deduce the associativity and commutativity of $\star$ in homology (cf. \cite{Rosenberg_Algebraic_K_Theory} pp.274-275 or \cite{LodayCyclic} pp.350-351). Here we give a direct proof (it means without the space $\mathrm{BGL}(R)^+$) which is essentially equivalent.
\begin{proof} The associativity and commutativity of $\star$ is a consequence of Lemma \ref{Lemma : relation between mu and oplus} and Lemma \ref{Lemma : commutativity of oplus} and of the invariance by conjugation of the group homology : Indeed, let $p,q,r \in \mathbb{N}^*$ and $(g_1,\dots,g_{p+q+r}) \in \mathrm{GL}(R)^{p+q+r}$. For all $1 \leq i \leq p+q+r$ let $A_{i}$ be an element in the class $g_{i}$. By adding $1's$ on the diagonal we can suppose that the $A_i's$ are all of the same size $n$. A representative in $\mathrm{GL}_{4n}^{p+q+r}(R)$ of $(g_1,\dots,g_p) \star \big((g_{p+1},\dots,g_{p+q}) \star (g_{p+q+1},\dots,g_{p+q+r})\big)$ is given by 
\begin{align*} \displaystyle
\sum_{\sigma \in \mathrm{Sh}_{p,q+r}} \sum_{\gamma \in \mathrm{Sh}_{q,r}} \epsilon(\sigma)\epsilon(\gamma) \, (a_{\beta(\sigma,\gamma)^{-1}(1)},\dots,a_{\beta(\sigma,\gamma)^{-1}(p+q+r)})
\end{align*}
where $a_i$ is equal to
\begin{itemize}
\item $\mu_{2n}(A_i \oplus \mathrm{I}_n,\mathrm{I}_n \oplus \mathrm{I}_n)$ for all $1 \leq i \leq p$,
\item $\mu_{2n}\big(\mathrm{I}_n \oplus \mathrm{I}_n,\mu_n(A_i,\mathrm{I}_n)\big)$ for all $p+1 \leq i \leq p+q$,
\item $\mu_{2n}\big(\mathrm{I}_n \oplus \mathrm{I}_n,\mu_n(\mathrm{I}_n,A_i)\big)$ for all $p+q+1 \leq i \leq p+q+r$.
\end{itemize} 
A representative in $\mathrm{GL}_{4n}^{p+q+r}(R)$ of $\big((g_1,\dots,g_p) \star (g_{p+1},\dots,g_{p+q})\big) \star (g_{p+q+1},\dots,g_{p+q+r})\big)$ is given by
\begin{align*}\displaystyle
\sum_{\sigma \in \mathrm{Sh}_{p+q,r}} \sum_{\gamma \in \mathrm{Sh}_{p,q}} \epsilon(\sigma)\epsilon(\gamma) \, (b_{\alpha(\sigma,\gamma)^{-1}(1)},\dots,b_{\alpha(\sigma,\gamma)^{-1}(p+q+r)})
\end{align*}
where $b_i$ is equal to
\begin{itemize}
\item $\mu_{2n}\big(\mu_n(A_i \oplus \mathrm{I}_n),\mathrm{I}_n \oplus \mathrm{I}_n\big)$ for all $1 \leq i \leq p$,
\item $\mu_{2n}\big(\mu_n(\mathrm{I}_n,A_i)\mathrm{I}_n \oplus \mathrm{I}_n\big)$ for all $p+1 \leq i \leq p+q$,
\item $\mu_{2n}\big(\mathrm{I}_n \oplus \mathrm{I}_n,A_i \oplus \mathrm{I}_n)$ for all $p+q+1 \leq i \leq p+q+r$.
\end{itemize} 
Lemma \ref{Lemma : relation between mu and oplus} and Lemma \ref{Lemma : commutativity of oplus}, imply that for all $n \in \mathbb{N}^*$ there are  matrices $X,Y \in \mathrm{GL}_{4n}(R)$ such that for all $A \in \mathrm{GL}_n(R)$ 
\begin{itemize}
\item $X^{-1}\mu_{2n}(A \oplus \mathrm{I}_n,\mathrm{I}_n \oplus \mathrm{I}_n)X = A \oplus \mathrm{I}_n \oplus \mathrm{I}_n \oplus \mathrm{I}_n = Y^{-1}\mu_{2n}(\mu_n(A,\mathrm{I}_n),\mathrm{I}_n \oplus \mathrm{I}_n)Y$,
\item $X^{-1}\mu_{2n}\big(\mathrm{I}_n \oplus \mathrm{I}_n,\mu_n(A,\mathrm{I}_n)\big)X = \mathrm{I}_n \oplus \mathrm{I}_n \oplus A \oplus \mathrm{I}_n = Y^{-1}\mu_{2n}(\mu_n(\mathrm{I}_n,A),\mathrm{I}_n \oplus \mathrm{I}_n)Y $,
\item $X^{-1}\mu_{2n}\big(\mathrm{I}_n \oplus \mathrm{I}_n,\mu_n(\mathrm{I}_n,A)\big)X = \mathrm{I}_n \oplus \mathrm{I}_n \oplus \mathrm{I}_n \oplus A =  Y^{-1}\mu_{2n}(\mathrm{I}_n \oplus \mathrm{I}_n,A \oplus \mathrm{I}_n)Y$.
\end{itemize} 
Then, using the change of variables $\mathrm{Sh}_{p,q+r} \times \mathrm{Sh}_{q,r} \stackrel{\alpha}{\simeq} \mathrm{Sh}_{p,q,r} \stackrel{\beta}{\simeq} \mathrm{Sh}_{p,q+r} \times \mathrm{Sh}_{p,q}$, the invariance by conjugation of the group homology imply the associativity.

\vskip 0.2cm

A representative in $\mathrm{GL}_{2n}^{p+q}(R)$ of $(g_1,\dots,g_p) \star (g_{p+1},\dots,g_{p+q})$ is given by
\begin{align*} \displaystyle
\sum_{\sigma \in \mathrm{Sh}_{p,q}} \epsilon(\sigma) (a_{\sigma^{-1}(1)},\dots,a_{\sigma^{-1}(p+q)})
\end{align*}
where $a_i$ is equal to
\begin{itemize}
\item $\mu_{2n}(A_i,\mathrm{I}_n)$ if $1 \leq i \leq p$,
\item $\mu_{2n}(\mathrm{I}_n,A_i)$ if $p+1 \leq i \leq p+q$.
\end{itemize}
A representative in $\mathrm{GL}_{2n}(R)$ of $(g_{p+1},\dots,g_{p+q}) \star (g_1,\dots,g_p)$ is given by
\begin{align} \displaystyle
\sum_{\gamma \in \mathrm{Sh}_{q,p}} \epsilon(\gamma) (b_{\gamma^{-1}(1)},\dots,b_{\gamma^{-1}(p+q)})
\end{align}
where $b_i$ is equal to
\begin{itemize}
\item $\mu_{2n}(\mathrm{I}_n,A_i)$ if $1 \leq i \leq p$,
\item $\mu_{2n}(A_i,\mathrm{I}_n)$ if $p+1 \leq i \leq p+q$.
\end{itemize}
Lemma \ref{Lemma : relation between mu and oplus} and Lemma \ref{Lemma : commutativity of oplus}, imply that for all $n \in \mathbb{N}^*$ there are  matrices $X,Y \in \mathrm{GL}_{2n}(R)$ such that for all $A \in \mathrm{GL}_n(R)$ 
\begin{align*}
X^{-1}\mu_{2n}(A, \mathrm{I}_n)X = A \oplus \mathrm{I}_n = Y^{-1} \mu_{2n}(\mathrm{I_n},A)Y,\\
X^{-1}\mu_{2n}(\mathrm{I}_n,A)X = \mathrm{I}_n \oplus A= Y^{-1} \mu_{2n}(A,\mathrm{I_n})Y.
\end{align*}
Then, using the change of variables $\mathrm{Sh}_{p,q} \stackrel{\iota}{\simeq} \mathrm{Sh}_{q,p}$, the invariance by conjugation of the group homology implies the commutativity.
\end{proof}
With this product $\mathrm{H}_{\bullet}(\mathrm{GL}(R),\Bbbk)$ is an associative and commutative algebra. Moreover group homology is naturally endowed with a cocommutative coalgebra (Theorem \ref{Theorem : commutative up to homotopy coalgebra structure on C(X)}). These algebra and coalgebra structures on $\mathrm{H}_{\bullet}(\mathrm{GL}(R),\Bbbk)$ are compatible, \textit{i.e.} they satisfy the \textit{Hopf relation}
\begin{align*}
\Delta \circ \star = \star_{\otimes} \circ (\Delta \otimes \Delta)
\end{align*} 
Therefore the group homology of $\mathrm{GL}(R)$ with trivial coefficients in a field $\Bbbk$ is a commutative Hopf algebra. This Hopf algebra is connected so if $\Bbbk$ is a field of characteristic $0$ the Hopf-Borel Theorem \ref{Theorem : Hopf-Borel} implies that the Hopf algebra $(\mathrm{H}_{\bullet}(\mathrm{GL}(R),\Bbbk),\Delta,\star)$ is free and cofree over its primitive part.

\subsection{A coZinbiel-associative bialgebra structure on the rack homology of the general linear group} The multiplication $\mu : \mathrm{GL}(R) \times \mathrm{GL}(R) \to \mathrm{GL}(R)$ being a group morphism, it is rack morphism. As a consequence the chain complex $\mathrm{CR}_{\bullet}(\mathrm{GL}(R),\Bbbk)$ computing the rack homology of $\mathrm{GL}(R)$ is provided with a product $\star$, still called the Pontryagin product, and defined by the same formula as before.
\begin{align*}
\star : \mathrm{CR}_{\bullet}(\mathrm{GL}(R),\Bbbk) \otimes \mathrm{CR}_{\bullet}(\mathrm{GL}(R),\Bbbk) \stackrel{\mathrm{EZ}}{\longrightarrow} \mathrm{CR}_{\bullet}(\mathrm{GL}(R) \times \mathrm{GL}(R),\Bbbk) \stackrel{\mathrm{CR}_{\bullet}(\mu)}{\longrightarrow} \mathrm{CR}_{\bullet}(\mathrm{GL}(R),\Bbbk). 
\end{align*}
An explicit formula for $\star$ on $\mathrm{CR}_{\bullet}(\mathrm{GL}(R),\Bbbk)$ is given by :
\begin{align*}
\mathrm{F} \star \mathrm{F}' = \mu \circ (\mathrm{F} \times \mathrm{F}') \circ \mathrm{i}_{p,q}
\end{align*}
where $i_{p,q}$ is the functor from $\square_{p+q}$ to $\square_{p} \times \square_q$ defined by $\mathrm{i}_{p,q}(\epsilon_1,\dots,\epsilon_{p+q}) = \big((\epsilon_1,\dots,\epsilon_p),(\epsilon_{p+1},\dots,\epsilon_{p+q})\big)$. Under the bijection $\mathrm{CR}_{\bullet}(\mathrm{GL}(R),\Bbbk) \simeq \Bbbk.\mathrm{GL}(R)^n$ the product $\star$ is equal to 
\begin{align*}
(g_1,\dots,g_p) \star (g_{p+1},\dots,g_{p+q}) = (g_1^{\mu},\dots,g_{p+q}^{\mu}).
\end{align*}
\begin{proposition}
The product $\star$ on $\mathrm{HR}_{\bullet}(\mathrm{GL}(R),\Bbbk)$ is associative.
\end{proposition} 
\begin{proof} The proof is the same proof as in Proposition \ref{Proposition : Pontryagin product is associative and commutative in group homology}. The only thing we have to prove is the invariance by conjugation of the rack homology.
\begin{lemma}\label{Lemma : invariance by conjugation of rack homology}
Let $X$ be a rack and $a \in X$. The conjugation map $c_a = - \lhd a$ induces the identity in homology.
\end{lemma}
Let $h_a : \mathrm{CR}_{\bullet}(X,\Bbbk) \to \mathrm{CR}_{\bullet}(X,\Bbbk)[1]$ be the map defined by 
$h_a(x_1,\dots,x_n) := (a,x_1,\dots,x_n)$.
This is a chain homotopy between the identity chain map of $\mathrm{CR}_{\bullet}(X,\Bbbk)$ and the chain map $\mathrm{CR}_{\bullet}(c_a)$. Therefore the rack homology of $X$ is invariant by conjugation by $a$. 
\end{proof}
In conclusion $\mathrm{HR}_{\bullet}(\mathrm{GL}(R),\Bbbk)$ is provided with an associative algebra structure and a coZinbiel coalgebra structure (Theorem \ref{Theorem : Zinbiel up to homotopy coalgebra structure on rack homology}). These structures being compatible, \textit{i.e.} they satisfy the \textit{semi-Hopf relation}
$\Delta_{\prec} \circ \star = \star_{\otimes} \circ (\Delta_{\prec} \otimes \Delta)$,
we proved the following theorem.
\begin{theorem}\label{Theorem : coZinbiel-associative bialgebra structure on the rack homology of the linear group}
$(\mathrm{HR}_{\bullet}(\mathrm{GL}(R),\Bbbk),\Delta_{\prec},\star)$ is a connected coZinbiel-associative bialgebra. As a consequence this bialgebra is free and cofree over its primitive part $\mathcal{P}$.
\begin{align*}
\mathrm{HR}_{\bullet}(\mathrm{GL}(R),\Bbbk) \simeq \mathrm{T}(\mathcal{P})
\end{align*}
\end{theorem}
This theorem is consistent with the conjecture of J.-L. Loday about the existence of a coZinbiel-associative bialgebra structure on the conjectural Leibniz homology of the linear group.

%

\appendix

\section{Acyclic models} This section is a summary of the Eilenberg-MacLane article \cite{EilenbergMacLane} on acyclic models theory. This theory provides a conceptual way to prove existence and unicity (up to homotopy) of chain morphisms between chain complexes. 

\subsection{Representable functor} \label{Appendix : Representable functor} Let $\mathbf{A}$ be a category, $\mathcal{M}$ be a set of \textit{models} (\textit{i.e.} a subset of the class of objects in $\mathbf{A}$), and $\mathrm{T}$ be a functor from $\mathbf{A}$ to $\Bbbk\mathbf{Mod}$. Let us denote by $\widetilde{\mathrm{T}}$ the functor from $\mathbf{A}$ to $\Bbbk\mathbf{Mod}$ defined by $\widetilde{\mathrm{T}}(A) = \Bbbk.\{(\phi,m) \, | \phi : M \to A \, , \, m \in M \, , \, \forall M \in \mathcal{M}\}$ on objects, and $\widetilde{\mathrm{T}}(f) = (f \circ \phi, m)$ on morphisms. There is a natural transformation $\Phi$ from $\widetilde{\mathrm{T}}$ to $\mathrm{T}$ defined by $\Phi_A(\phi,m) = \mathrm{T}(\phi)(m)$. The functor $\mathrm{T}$ is said \textit{representable} if there is a natural transformation $\Psi$ from $\mathrm{T}$ to $\widetilde{\mathrm{T}}$ satisfying $\Phi \circ \Psi = \mathrm{id}$.

\begin{lemma} \label{the quotient of a representable is representable}
Let $\mathbf{A}$ be a category, $\mathrm{T}$ and $\mathrm{T}_1$ be functors from $\mathbf{A}$ to $\mathbf{Mod}$, and $\xi : \mathrm{T} \to \mathrm{T}_1$ and $\eta : \mathrm{T}_1 \to \mathrm{T}$ be natural transformations such that $\xi \circ \eta = \mathrm{id}$. If $\mathrm{T}$ is representable then so is $\mathrm{T_1}$.
\end{lemma}

\subsection{Representability of the evaluation functor}Let $\mathbf{A}$ be a category and $A$ be an object of $\mathbf{A}$. The \textit{evaluation functor at $A$} is the functor $\mathrm{ev}_{A}$ from $[\mathbf{C}^{op},\Bbbk\mathbf{Mod}]$ to $\Bbbk\mathbf{Mod}$ defined on objects by $\mathrm{ev}_{A}(\mathrm{X}) = \mathrm{X}(A)$ and on morphisms by $\mathrm{ev}_{A}(\tau) = \tau_A$. 

\begin{proposition}
For all object $A \in \mathbf{A}$ the evaluation functor $\mathrm{ev}_A$ is representable (with set of models the set with one element $\{\Bbbk.\mathrm{Hom}_{\mathbf{A}}(-,A)\}$).
\end{proposition}

\begin{proof}
Take as set of models the set with one element $\mathcal{M} := \{\Bbbk.\mathrm{Hom}_{\mathbf{A}}(-,A)\}$. Let us define a natural transformation $\Psi$ from $\mathrm{ev}_A$ to $\widetilde{\mathrm{ev}}_{\mathrm{A}}$ by $\Psi_X(x) = (\phi_x,\mathrm{id}_A)$ where $\phi_x$ is the unique natural transformation from $\mathrm{Hom}_{\mathbf{A}}(-,A)$ to $\mathrm{X}$ satisfying $\phi_x(\mathrm{id}_A) = x$ (\textit{Yoneda's lemma}). By definition we have $\Phi \circ \Psi$ equal to the identity, so the functor $\mathrm{ev}_A$ is representable.
\end{proof}

\subsection{Map and homotopy} Let $\mathbf{A}$ be a category, $\mathbf{Ch^+}$ be the category of chain complexes of $\Bbbk$-modules and chain maps, and $\mathrm{K}$ be a functor from $\mathbf{A}$ to $\mathbf{Ch^+}$. For each object $\mathrm{A} \in \mathbf{A}$, the functor $\mathrm{K}$ determines a complex $\mathrm{K_{\bullet}(A)}$ composed of modules $\mathrm{K}_q(\mathrm{A})$ and differentials $d^q : \mathrm{K}_q(\mathrm{A}) \to \mathrm{K}_{q-1}(\mathrm{A})$ with $d^{q-1}d^q = 0$. The modules $\mathrm{K}_q(\mathrm{A})$ yield a functor $\mathrm{K}_q$ from $\mathbf{A}$ to $\Bbbk\mathbf{Mod}$ and the differentials yield natural transformations $d^q : \mathrm{K}_q \to \mathrm{K}_{q-1}$ with $d^{q-1}d^q = 0$.

\vskip 0.2cm

Let $\mathrm{K}$ and $\mathrm{L}$ be two functors from $\mathbf{A}$ to $\mathbf{Ch^+}$. A \textit{map} $\mathrm{f : K \to L}$ is a family of a natural transformations $\mathrm{f}_q : \mathrm{K}_q \to \mathrm{L}_q$ such that $d_q\mathrm{f}_q = \mathrm{f}_{q-1}d_q$. If $\mathrm{f}_q$ is defined and satisfies this equation only for $q \leq n$, we say that $\mathrm{f}$ is a \textit{map in dimensions $\leq n$}.

\vskip 0.2cm

Let $\mathrm{f,g : K \to L}$ be two maps. A \textit{homotopy} $\mathrm{D}$ from $\mathrm{f}$ to $\mathrm{g}$ is a sequence of natural transformations $\mathrm{D}_q : \mathrm{K}_q \to \mathrm{L}_{q+1}$ satisfying :
\begin{align*}
d^{q+1}\mathrm{D}_q + \mathrm{D}_{q-1}d^q = \mathrm{g}_q - \mathrm{f}_q.
\end{align*}
If the maps $\mathrm{D}_q$ are defined and satisfy this equality only for $q \leq n$, we say that $\mathrm{D}$ is a homotopy in dimensions $\leq n$.

\subsection{Acyclic models theorems} The two fundamental acyclic models theorems are the following. The first theorem concerns extension of morphisms whereas the second theorem concerns extension of homotopies between morphisms.

\begin{theorem}\label{Theorem : acyclic models theorem 1}
Let $\mathrm{K}$ and $\mathrm{L}$ be functors from a category $\mathbf{A}$ to the category $\mathbf{Ch}^+$, and let $\mathrm{f} : \mathrm{K} \to \mathrm{L}$ be a map in dimensions $< q$. If $\mathrm{K}_q$ is representable and if $\mathrm{H}_{q-1}\big(\mathrm{L}(M)\big) = 0$ for each model $M \in \mathcal{M}$, then $\mathrm{f}$ admits an extension to a map $\mathrm{K} \to \mathrm{L}$ in dimension $\leq q$. 
\end{theorem}

\begin{theorem}\label{Theorem : acyclic models theorem 2}
 Let $\mathrm{K}$ and $\mathrm{L}$ be functors from a category $\mathbf{A}$ to the category $\mathbf{Ch}^+$, let $\mathrm{f,g} : \mathrm{K} \to \mathrm{L}$ be maps, and let $\mathrm{D : f \simeq g}$ be a homotopy in dimensions $< q$. If $\mathrm{K}_q$ is representable and if $\mathrm{H}_{q}\big(\mathrm{L}(M)\big) = 0$ for each model $M \in \mathcal{M}$, then $\mathrm{D}$ admits an extension to a homotopy $\mathrm{f \simeq g}$  in dimension $\leq q$. 
\end{theorem}

\section{Simplicial and Cubical sets} \label{Appendix : Simplicial and cubical sets} This appendix is a reminder of basics of simplicial and cubical sets theory. It is based on Appendix B of \cite{LodayCyclic} for the simplicial set theory. 

\subsection{The category $\Delta$} Let $\Delta$ be the small category with set of objects $\{\Delta_n := \{0,\dots,n\}\}_{n \in \mathbb{N}}$, and set of morphisms the set of non decreasing maps. Morphisms in $\Delta$ are generated by the two families of maps $\{\delta_i : \Delta_{n-1} \to \Delta_n\}_{0 \leq i \leq n}$ and $\{\sigma_i : \Delta_{n} \to \Delta_{n-1}\}_{0 \leq i \leq n}$ defined by
\begin{align*}
\delta_i(j) := \left\{
\begin{array}{ll}
j & \text{if } j < i, \\
j+1 & \text{if } j \geq i.
\end{array}
\right.
\, \text{ and } \,
\sigma_i(j) := \left\{
\begin{array}{ll}
j & \text{if } j \leq i, \\
j-1 & \text{if } j > i.
\end{array}
\right.
\end{align*}
These two families of maps satisfy relations called \textit{cosimplicial identities}:
\begin{align*}
\delta_i \delta_j &= \delta_{j+1} \delta_i \quad \forall i \leq j,\\
\sigma_i\sigma_j &= \sigma_{j-1}\sigma_i \quad \forall i < j,\\
\sigma_i\delta_j &= \left\{
\begin{array}{ll}
\delta_{j-1}\sigma_i & \text{if } i < j ,\\
\mathrm{id} & \text{if } i = j , \\
\delta_j,\sigma_{i-1} 	& \text{if } i > j.
\end{array}
\right.
\end{align*}

\subsection{Simplicial sets} A \textit{simplicial set} is a functor $\mathrm{X}$ from $\Delta^{\mathrm{op}}$ to $\mathbf{Set}$. Equivalently, a simplicial set is a family of sets $\{\mathrm{X}_n\}_{n \in \mathbb{N}}$ with two families of maps $\{d_{i} : \mathrm{X}_n \to \mathrm{X}_{n-1}\}_{0 \leq i \leq n}$ and $\{s_i : \mathrm{X}_{n-1} \to \mathrm{X}_n\}_{0 \leq i \leq n}$ satisfying relations called \textit{simplicial identities} :
\begin{align*}
 d_{j} d_{i} &= d_{i} d_{j+1} \quad \forall i \leq j,
 \\
 s_j s_i &= s_{i-1} s_{j} \quad \forall i < j,
 \\
 d_{j} s_i&=
\left\{\begin{array}{ll}
 s_i d_{j-1}& \text{if } i < j,\\
\mathrm{id} & \text{if } i = j,\\
 s_{i-1} d_{j}& \text{if } i > j.
\end{array}\right. 
\end{align*}
\par
Given two simplicial sets $\mathrm{X}$ and $\mathrm{Y}$, a \textit{morphism of simplicial sets} from $\mathrm{X}$ to $\mathrm{Y}$ is a natural transformation $\mathrm{t} : \mathrm{X} \to \mathrm{Y}$. This is equivalent to a family of set theoretical maps $\{\mathrm{t}_n : \mathrm{X}_n \to \mathrm{Y}_n\}_{n \in \mathbb{N}}$ satisfying the following commutativity relations :
\begin{align*}
\mathrm{t}_{n-1}d^{\mathrm{X}}_{i} &= d^{\mathrm{Y}}_{i}\mathrm{t_n},\\ 
\mathrm{t}_{n}s^{\mathrm{X}}_i &= s^{\mathrm{Y}}_{i}\mathrm{t_{n-1}}.
\end{align*} 
The category of simplicial sets is denoted by $\mathbf{sSet}$. 
\vskip 0.2cm
By duality a \textit{cosimplicial set} is a functor $\mathrm{X}$ from $\Delta$ to $\mathbf{Set}$. This is equivalent to the data of a family of sets $\{X_n\}_{n \in \mathbb{N}}$ with two families of maps $\{d^{i} : \mathrm{X}_{n-1} \to \mathrm{X}_{n}\}_{0 \leq i \leq n}$ and $\{s^i : \mathrm{X}_{n} \to \mathrm{X}_{n-1}\}_{0 \leq i \leq n}$ satisfying the \textit{cosimplicial identities}.

\begin{example} $\Delta^n := \mathrm{Hom}_{\Delta}(-,\Delta_n) : \Delta^{\mathrm{op}} \to \mathbf{Set}$ for all $n \in \mathbb{N}$ is a simplicial set.
\end{example}

\begin{example} Let us denote by $|\Delta^n|$ the "classical" $n$-simplex in $\mathbb{R}^{n+1}$, \textit{i.e.} the convex hull of the points $(0,\dots,1,\dots,0) \in \mathbb{R}^{n+1}$. We define a cosimplicial topological space $|\Delta| : \Delta \to \mathbf{Top}$ by $|\Delta|_n := |\Delta^n|$, $d^i(t_1,\dots,t_{n}) = (t_1,\dots,t_{i-1},0,t_i,\dots,t_n)$ and $s^i(t_1,\dots,t_n) = (t_1,\dots,\widehat{t_1},\dots,t_n)$.
\end{example}

\begin{example} \label{Example : Singular functor} Let $X$ be a topological space. The simplicial set $\mathrm{Sing}(X)$ is defined by the composition $\mathrm{Sing}(\mathrm{X}) := \mathrm{Hom}_{\mathbf{Top}}(-,\mathrm{X}) \circ |\Delta|$. This construction defines a functor $\mathrm{Sing} : \mathbf{Top} \to \mathbf{sSet}$.
\end{example}

\subsection{Simplicial geometric realization} The \textit{geometric realization $|\mathrm{X}|$} of a simplicial set $\mathrm{X}$ is the coend of the functor $\mathrm{X} \times |\Delta| : \Delta^{\mathrm{op}} \times \Delta \to \mathbf{Top}$ defined on objects by $(\mathrm{X} \times |\Delta|)(\Delta_m,\Delta_n) := \mathrm{X}_m \times |\Delta^n|$.
\begin{align*}
|\mathrm{X}| := \int^n \mathrm{X}_n \times |\Delta^n|
\end{align*}
The topological space $|\mathrm{X}|$ is the quotient of the topological space $\coprod_{n \in \mathbb{N}} \mathrm{X}_n \times |\Delta^n|$ ($\mathrm{X}_n$ provided with the discrete topology) by the equivalence relation 
\begin{align*}
\big(\mathrm{X}(f)x,\epsilon\big) \simeq \big(x,|\Delta|(f)(\epsilon)\big).
\end{align*}
This construction induces a functor $|-| : \mathbf{sSet} \to \mathbf{Top}$ left adjoint to the functor $\mathrm{Sing}: \mathbf{Top} \to \mathbf{sSet}$ defined in Example \ref{Example : Singular functor}.
\subsection{Homology of simplicial sets} Let $\mathrm{X}$ be a simplicial set. For all $n \in \mathbb{N}$ let us denote by $\mathrm{Q}_{n}(\mathrm{X})$ the free module over $\Bbbk$ generated by $\mathrm{X}_n$, and by $\mathrm{C}_{n}(\mathrm{X})$ the quotient of $\mathrm{Q}_{n}(\mathrm{X})$ by the subspace generated by the images of the degeneracies $\mathrm{D}_n(\mathrm{X}) := \Bbbk.\{\mathrm{Im}(s_i) \, | \, i = 0,\dots,n\}$. These constructions are natural in $\mathrm{X}$ and so induce functors
\begin{align*}
\mathrm{Q}_{n} : \mathbf{sSet} \to \mathbf{\Bbbk Mod} \, \text{ and } \, \mathrm{C}_{n} : \mathbf{sSet} \to \mathbf{\Bbbk Mod}. 
\end{align*} 

\begin{lemma} \label{lemma : Cndelta is representable} 
For all $n \in \mathbb{N}$ the functor $\mathrm{Q}_n : \mathbf{sSet} \to \Bbbk \mathbf{Mod}$ is representable (by the set of models $\{\Bbbk.\Delta^n\}$).
\end{lemma}

\begin{proof}
Let $n \in \mathbb{N}$ be fixed and take as set of models $\mathcal{M}$ the set with one element $\{\Bbbk.\Delta^n\}$. By definition the functor $\mathrm{Q}_n$ is the composition of the functor $\Bbbk. : \mathbf{Set} \to \Bbbk\mathbf{Mod}$ and the evaluation functor $\mathrm{ev}_{\Delta_n} : \mathbf{cSet} \to \mathbf{Set}$
\begin{align*}
\mathrm{Q}_n := \Bbbk. \circ \mathrm{ev}_{\square_n}. 
\end{align*}
Let us define a natural transformation $\Psi$ from $\mathrm{Q}_n$ to $\mathrm{\widetilde{Q}}_n$ by the formula
\begin{align*}
\Psi_{\mathrm{X}} : \mathrm{Q}_n(\mathrm{X}) \to \mathrm{\widetilde{Q}}_n(\mathrm{X}) \, ; \, \Psi_{\mathrm{X}}(x) = (\phi_x,\mathrm{id}_{\Delta_n}),
\end{align*}
where $\phi_x$ is the unique natural transformation from $\square^n$ to $\mathrm{X}$ such that $(\phi_x)_n(\mathrm{id}_{\Delta_n}) = x$ (Yoneda's Lemma). We have $\Phi\circ \Psi = \mathrm{id}$ so $\mathrm{Q}_n$ is representable.
\end{proof}

Thanks to the simplicial identities the degree $-1$ graded map $ d := \sum_{i=0}^n (-1)^{i} d_{i}$ satisfies the equation $d^2 = 0$. This constructions is natural in $\mathrm{X}$ and so induces a functor
\begin{align*}
\mathrm{Q}_{\bullet} : \mathbf{cSet} \to \mathbf{Ch}^+ 
\end{align*} 
\par
The \textit{(simplicial) homology of} $\mathrm{X}$, denoted $\mathrm{H_{\bullet}}(\mathrm{X})$, is the homology of the chain complex $\mathrm{Q}_{\bullet}(\mathrm{X})$.

\begin{lemma} \label{lemma : deltan is contractible}
For all $n \in \mathbb{N}$, $\mathrm{H}_{p}(\Delta^n,\Bbbk) =
\left\{\begin{array}{ll}
\Bbbk & \text{if } p = 0,\\
0 & \text{if } p >0 .
\end{array}\right. 
$.
\end{lemma}

\begin{proof}
Let $n \in \mathbb{N}$ be fixed. The simplicial homology of $\Delta^n$ is isomorphic to the singular homology of its geometric realization $|\Delta^n|$. The topological space $|\Delta^n|$ is contractible so $\mathrm{H}_{\bullet}(\Delta^n,\Bbbk) =  0$.
\end{proof}

\subsection{The category $\square$} Let $\square$ be the small category with set of objects the sets $\{\square_n := \{0,1\}^{n}\}_{n \in \mathbb{N}}$, and set of morphisms the set of maps generated by the two families $\{\delta_{i,\epsilon} : \square_{n-1} \to \square_n\}_{\epsilon \in \{0,1\},1 \leq i \leq n}$ and $\{\sigma_i : \square_n \to \square_{n-1}\}_{1 \leq i \leq n}$ where
\begin{align*}
\delta_{i,\epsilon}(\epsilon_1,\dots,\epsilon_{n-1}) &:= (\epsilon_1,\dots,\epsilon_{i-1},\epsilon,\epsilon_i,\dots,\epsilon_{n-1}),\\
\sigma_i(\epsilon_1,\dots,\epsilon_n) &:= (\epsilon_1,\dots,\epsilon_{i-1},\epsilon_{i+1},\dots,\epsilon_n). 
\end{align*}
These two families of maps satisfy relations called \textit{cocubical identities} :
\begin{align*}
 \delta_{i,\epsilon} \delta_{j,\omega} &= \delta_{j+1,\omega} \delta_{i,\epsilon} \quad \forall i \leq j,
 \\
 \sigma_i \sigma_j &= \sigma_{j-1} \sigma_{i} \quad \forall i < j,
 \\
 \sigma_i \delta_{j,\epsilon} &=
\left\{\begin{array}{ll}
\delta_{j-1,\epsilon} \sigma_i & \text{if } i < j,\\
\mathrm{id} & \text{if } i = j,\\
\delta_{j,\epsilon} \sigma_{i-1} & \text{if } i > j.
\end{array}\right. 
\end{align*}
Therefore each map $f \in \mathrm{Hom}_{\square}(\square_m,\square_n)$ can be rewritten in a unique way :
\begin{align*}
f = \sigma_{j_1}\cdots\sigma_{j_q}\delta_{i_p,\epsilon_p}\cdots\delta_{i_1,\epsilon_1},
\end{align*}
with $p+q = n-m, \, i_1 < \dots < i_p$ and $j_1 < \dots < j_q$.

\subsection{Cubical sets} \label{Definition : cubical sets} A \textit{cubical set} is a functor $\mathrm{X}$ from $\square^{\mathrm{op}}$ to $\mathbf{Set}$. Equivalently, a cubical set is a family of sets $\{\mathrm{X}_n\}_{n \in \mathbb{N}}$ with two families of maps $\{d_{i,\epsilon} : \mathrm{X}_n \to \mathrm{X}_{n-1}\}_{\epsilon \in \{0,1\}, 1 \leq i \leq n}$ and $\{s_i : \mathrm{X}_{n-1} \to \mathrm{X}_n\}_{1 \leq i \leq n}$ satisfying relations called \textit{cubical identities} :
\begin{align*}
 d_{j,\omega} d_{i,\epsilon} &= d_{i,\epsilon} d_{j+1,\omega} \quad \forall i \leq j,
 \\
 s_j s_i &= s_{i-1} s_{j} \quad \forall i < j,
 \\
 d_{j,\epsilon} s_i&=
\left\{\begin{array}{ll}
 s_i d_{j-1,\epsilon}& \text{if } i < j,\\
\mathrm{id} & \text{if } i = j,\\
 s_{i-1} d_{j,\epsilon}& \text{if } i > j.
\end{array}\right. 
\end{align*}
\par
Given two cubical sets $\mathrm{X}$ and $\mathrm{Y}$, a \textit{morphism of cubical sets} from $\mathrm{X}$ to $\mathrm{Y}$ is a natural transformation $\mathrm{t} : \mathrm{X} \to \mathrm{Y}$. This is equivalent to a family of set theoretical maps $\{\mathrm{t}_n : \mathrm{X}_n \to \mathrm{Y}_n\}_{n \in \mathbb{N}}$ satisfying the following commutativity relations :
\begin{align*}
\mathrm{t}_{n-1}d^{\mathrm{X}}_{i,\epsilon} &= d^{\mathrm{Y}}_{i,\epsilon}\mathrm{t_n},\\ 
\mathrm{t}_{n}s^{\mathrm{X}}_i &= s^{\mathrm{Y}}_{i}\mathrm{t_{n-1}}.
\end{align*} 
The category of cubical sets is denoted by $\mathbf{cSet}$.
\vskip 0.2cm
By duality a \textit{cocubical set} is a functor $\mathrm{X}$ from $\square$ to $\mathbf{Set}$. This is equivalent to the data of a family of sets $\{X_n\}_{n \in \mathbb{N}}$ with two families of maps $\{d^{i,\epsilon} : \mathrm{X}_{n-1} \to \mathrm{X}_{n}\}_{\epsilon \in \{0,1\},1 \leq i \leq n}$ and $\{s^i : \mathrm{X}_{n} \to \mathrm{X}_{n-1}\}_{1 \leq i \leq n}$ satisfying the \textit{cocubical identities}.
\begin{example} $\square^n := \mathrm{Hom}_{\square}(-,\square_n) : \square^{\mathrm{op}} \to \mathbf{Set}$ for all $n \in \mathbb{N}$.
\end{example}

\begin{example} Let us denote by $|\square^n|$ the "classical" $n$-cube in $\mathbb{R}^{n}$, \textit{i.e.} the topological space $[0,1]^n$. We define a cocubical topological space $|\square| : \square \to \mathbf{Top}$ by $|\square|_n := |\square^n|$, $d^{i,\epsilon}(t_1,\dots,t_{n}) = (t_1,\dots,t_{i-1},\epsilon,t_i,\dots,t_n)$ and $s^i(t_1,\dots,t_n) = (t_1,\dots,\widehat{t_1},\dots,t_n)$.
\end{example}

\begin{example} \label{Example : Cubical singular functor} Let $X$ be a topological space. The cubical set $\mathrm{Cub}(X)$ is defined by the composition $\mathrm{Cub}(X)_n := \mathrm{Hom}_{\mathbf{Top}}(-,X) \circ |\square|$. This construction defines a functor $\mathrm{Cub} : \mathbf{Top} \to \mathbf{cSet}$.
\end{example}

\subsection{Cubical geometric realization} The \textit{geometric realization $|\mathrm{X}|$} of a cubical set $\mathrm{X}$ is the coend of the functor $\mathrm{X} \times |\square| : \square^{\mathrm{op}} \times \square \to \mathbf{Top}$ defined on objects by $(\mathrm{X} \times |\square|)(\square_m,\square_n) := \mathrm{X}_m \times |\square^n|$.
\begin{align*}
|\mathrm{X}| := \int^n \mathrm{X}_n \times |\square^n|
\end{align*}
The topological space $|\mathrm{X}|$ is the quotient of the topological space $\coprod_{n \in \mathbb{N}} \mathrm{X}_n \times |\square^n|$ ($\mathrm{X}_n$ provided with the discrete topology) by the equivalence relation 
\begin{align*}
\big(\mathrm{X}(f)x,\epsilon\big) \simeq \big(x,|\square|(f)(\epsilon)\big).
\end{align*}
This construction induces a functor $|-| : \mathbf{cSet} \to \mathbf{Top}$ left adjoint to the functor $\mathrm{Cub}: \mathbf{Top} \to \mathbf{cSet}$ defined in Example \ref{Example : Cubical singular functor}.

\subsection{Homology of cubical sets} \label{Appendix : Homology of cubical sets} Let $\mathrm{X}$ be a cubical set. For all $n \in \mathbb{N}$ let us denote by $\mathrm{Q}_{n}(\mathrm{X})$ the free module over $\Bbbk$ generated by $\mathrm{X}_n$, and by $\mathrm{C}_{n}(\mathrm{X})$ the quotient of $\mathrm{Q}_{n}(\mathrm{X})$ by the subspace generated by the images of the degeneracies $\mathrm{D}_n(\mathrm{X}) := \Bbbk.\{\mathrm{Im}(s_i) \, | \, i = 1,\dots,n\}$. These constructions are natural in $\mathrm{X}$ and so induce functors
\begin{align*}
\mathrm{Q}_{n} : \mathbf{cSet} \to \mathbf{\Bbbk Mod} \, \text{ and } \, \mathrm{C}_{n} : \mathbf{cSet} \to \mathbf{\Bbbk Mod}. 
\end{align*} 
\begin{proposition} \label{Proposition : Cn representable}
The functors $\mathrm{Q}_n$ and $\mathrm{C}_n$ are representable by $\Bbbk.\square^n$.
\end{proposition}

\begin{proof}
Let $n \in \mathbb{N}$ be fixed and take as set of models $\mathcal{M}$ the set with one element $\{\Bbbk.\square^n\}$. By definition the functor $\mathrm{Q}_n$ is the composition of the functor $\Bbbk. : \mathbf{Set} \to \Bbbk\mathbf{Mod}$ and the evaluation functor $\mathrm{ev}_{\square_n} : \mathbf{cSet} \to \mathbf{Set}$
\begin{align*}
\mathrm{Q}_n := \Bbbk. \circ \mathrm{ev}_{\square_n}. 
\end{align*}
Let us define a natural transformation $\Psi$ from $\mathrm{Q}_n$ to $\mathrm{\widetilde{Q}}_n$ by the formula
\begin{align*}
\Psi_{\mathrm{X}} : \mathrm{Q}_n(\mathrm{X}) \to \mathrm{\widetilde{Q}}_n(\mathrm{X}) \, ; \, \Psi_{\mathrm{X}}(x) = (\phi_x,\mathrm{id}_{\square_n}),
\end{align*}
where $\phi_x$ is the unique natural transformation from $\square^n$ to $\mathrm{X}$ such that $(\phi_x)_n(\mathrm{id}_{\square_n}) = x$ (Yoneda's Lemma). We have $\Phi\circ \Psi = \mathrm{id}$ so $\mathrm{Q}_n$ is representable.
\vskip 0.2cm
The representability of $\mathrm{C}_n$ is a consequence of Lemma \ref{the quotient of a representable is representable}. Indeed, let $\xi : \mathrm{Q}_n \to \mathrm{C}_n$ be the natural transformation defined by passing to the quotient, and let $\eta : \mathrm{C}_n \to \mathrm{Q}_n$ be the natural transformation defined by 
\begin{align*}
\eta_{\mathrm{X}} := (\mathrm{id} - s_1 d_{1,0}) \cdots  (\mathrm{id} - s_n d_{n,0}). 
\end{align*} 
Thanks to the cubical identities this map sends $\mathrm{D}_n(\mathrm{X})$ to $0$ and so is well defined. Moreover $\xi \circ \eta = \mathrm{id}$ then by Lemma \ref{the quotient of a representable is representable} the functor $\mathrm{C}_n$ is representable. 
\end{proof}

Thanks to the cubical identities the degree $-1$ graded map $ d := \sum_{i=1}^n (-1)^{i+1} (d_{i,1}-d_{i,0})$ satisfies the equation $d^2 = 0$ and sends $\mathrm{D}_n(\mathrm{X})$ into $\mathrm{D}_{n-1}(\mathrm{X})$. These constructions are natural in $\mathrm{X}$ and so induce functors
\begin{align*}
\mathrm{Q}_{\bullet} : \mathbf{cSet} \to \mathbf{Ch}^+ \, \text{ and } \, \mathrm{C}_{\bullet} : \mathbf{cSet} \to \mathbf{Ch}^+. 
\end{align*} 
\par
The \textit{unormalized (cubical) homology of} $\mathrm{X}$, denoted $\mathrm{H^{u}_{\bullet}}(\mathrm{X})$, is the homology of the chain complex $\mathrm{Q}_{\bullet}(\mathrm{X})$. The \textit{(cubical) homology of} $\mathrm{X}$, denoted $\mathrm{H_{\bullet}}(\mathrm{X})$, is the homology of the chain complex $\mathrm{C}_{\bullet}(\mathrm{X})$.  

\begin{lemma} \label{lemma : squaren is contractible}
For all $n \in \mathbb{N}$, $\mathrm{H}_{p}(\square^n,\Bbbk) =
\left\{\begin{array}{ll}
\Bbbk & \text{if } p = 0,\\
0 & \text{if } p >0 .
\end{array}\right. 
$.
\end{lemma}

\begin{proof}
Let $n \in \mathbb{N}$ be fixed. The cubical homology of $\square^n$ is isomorphic to the singular homology of its geometric realization $[0,1]^n$. The topological space $[0,1]^n$ is contractible so $\mathrm{H}_{\bullet}(\square^n,\Bbbk) = 0$.
\end{proof}

\subsection{A differential graded coassociative coalgebra structure on the chain complex of a cubical set} Let $\mathrm{X}$ be a cubical set. Let us define a degree $0$ map $\Delta(\mathrm{X})$ from $\mathrm{C}_{\bullet}(\mathrm{X})$ to $\mathrm{C}_{\bullet}(\mathrm{X}) \otimes \mathrm{C}_{\bullet}(\mathrm{X})$ by the following formula :
\begin{align*}\displaystyle
\Delta(\mathrm{X})_n(x) := \bigoplus_{p+q = n} \sum_{\sigma \in \mathrm{Sh}_p,q} \epsilon(\sigma) \, (d_{\sigma(p+1),0}\cdots d_{\sigma(p+q),0})(x) \otimes (d_{\sigma(1),1}\cdots d_{\sigma(p),1})(x)
\end{align*}
for all $x \in \mathrm{X}_n, \, n \geq 1$ and $\Delta(\mathrm{X})_{0}(x) = x \otimes x$ for $x \in \mathrm{X}_0$. With enough stamina it is possible to check that this formula defines a chain complex morphism. A more conceptual way to prove that such a chain morphism exists and is unique up to homotopy is to use the method of acyclic models : Indeed, define $\Delta$ in dimension $0$ by the previous formula. It induces a map between homology groups $\Delta_0 : \mathrm{H}_{0}(X,\Bbbk) \to \mathrm{H}_0(X,\Bbbk)$. By induction we can extend this map to a unique up to homotopy morphism of chain complexes using the theorem of acyclic models. Indeed, by Proposition \ref{Proposition : Cn representable} the functor $\mathrm{C}_{n}$ is representable by $\square^{n}$ for all $n \in \mathbb{N}$, and by Lemma \ref{lemma : squaren is contractible} the homology groups $\mathrm{H}_{p}\big(\mathrm{C}_{\bullet}(\square^{n}) \otimes \mathrm{C}_{\bullet}(\square^{n})\big)$ vanish for all $p,n \in \mathbb{N}^*$. Then by  the theorem of acyclic models \ref{Theorem : acyclic models theorem 1} and \ref{Theorem : acyclic models theorem 2}, there exists a homotopy unique chain map $\Delta$ extending $\Delta_0$. 
  
\begin{theorem}\label{Theorem : commutative up to homotopy coalgebra structure on C(X)}
Let $X$ be a cubical set. Then $(\mathrm{C_{\bullet}(X,\Bbbk)},\Delta)$ is an homotopy coassociative and homotopy cocommutative coalgebra (in the category of chain complexes).
\end{theorem}

\begin{proof} The maps $(\mathrm{id} \otimes \Delta) \circ \Delta$ and $(\Delta \otimes \mathrm{id}) \circ \Delta$ are two maps from $\mathrm{C_{\bullet}}$ to $\mathrm{C_{\bullet} \otimes C_{\bullet} \otimes C_{\bullet}}$ which coincide in dimension $0$. Therefore by the theorem of acyclic models \ref{Theorem : acyclic models theorem 2} these two maps are homotopic. 
\vskip 0.2cm
The maps $\Delta$ and $\tau \circ  \Delta$ are two maps from $\mathrm{C_{\bullet}}$ to $\mathrm{C_{\bullet} \otimes C_{\bullet}}$ which coincide in dimension $0$. Therefore by the theorem of acyclic models \ref{Theorem : acyclic models theorem 2} these two maps are homotopic. 
\end{proof}

\begin{corollary} Let $\mathrm{X}$ be a cubical set and $\Bbbk$ be a field. Then $(\mathrm{H}_{\bullet}(\mathrm{X},\Bbbk),\Delta)$ is a coassociative and cocommutative coalgebra (in the category of graded vector space). 
\end{corollary}

\bibliographystyle{alpha}
\bibliography{bibliographie}

\end{document}